\documentclass[10pt]{amsart}
\usepackage{amssymb,graphicx}
\usepackage{mathrsfs, bbm, slashed}
\usepackage{color}
\usepackage{enumerate}
\usepackage{a4wide}
\usepackage[colorlinks=true,citecolor=blue,linkcolor=blue]{hyperref}

\newcommand{\R}{\mathbb{R}}

\newcommand{\C}{\mathbb{C}}

\newcommand{\Z}{\mathbb{Z}}

\newcommand{\N}{\mathbb{N}}
\newcommand{\T}{\mathbb{T}}

\newcommand{\bS}{\mathbb{S}}

\newcommand{\sD}{\mathscr{D}}

\newcommand{\sF}{\mathscr{F}}
\newcommand{\sL}{\mathscr{L}}

\newcommand{\sH}{\mathscr{H}}
\newcommand{\sK}{\mathscr{K}}
\newcommand{\sM}{\mathscr{M}}
\newcommand{\sMs}{\mathscr{M}_{\textup{s}}}
\newcommand{\sS}{\mathscr{S}}

\newcommand{\sV}{\mathscr{V}}
\newcommand{\sW}{\mathscr{W}}

\newcommand{\fM}{\mathfrak{M}}

\newcommand{\co}{\mathfrak{c}_{0}}

 \newcommand{\psido}{$\Psi$DO}
 \newcommand{\psidos}{$\Psi$DOs}

\def\Xint#1{\mathchoice
{\XXint\displaystyle\textstyle{#1}}%
{\XXint\textstyle\scriptstyle{#1}}%
{\XXint\scriptstyle\scriptscriptstyle{#1}}%
{\XXint\scriptscriptstyle\scriptscriptstyle{#1}}%
\!\int}
\def\XXint#1#2#3{{\setbox0=\hbox{$#1{#2#3}{\int}$}
\vcenter{\hbox{$#2#3$}}\kern-.5\wd0}}

\def\dashint{\Xint-}

\newcommand{\bint}{\ensuremath{\dashint}}

\newcommand{\op}{\operatorname} 
\newcommand{\limw}{{\lim}_{\omega}\,}

\newcommand{\tr}{\op{tr}}

\newcommand{\Tr}{\op{Tr}}
\newcommand{\Trw}{\Tr_{\omega}}

\newcommand{\ran}{\op{ran}}

\newcommand{\Com}{\op{Com}}
\newcommand{\End}{\op{End}}

\newcommand{\Res}{\op{Res}}

\newcommand{\scal}[2]{\ensuremath{\left\langle #1 | #2 \right\rangle}} 
\newcommand{\acou}[2]{\ensuremath{\left\langle #1 , #2 \right\rangle}} 
 
\newcommand{\bigacou}[2]{\ensuremath{\big\langle #1 , #2 \big\rangle}}

\newcommand{\supp}{\op{supp}}

\def\Xint#1{\mathchoice
{\XXint\displaystyle\textstyle{#1}}%
{\XXint\textstyle\scriptstyle{#1}}%
{\XXint\scriptstyle\scriptscriptstyle{#1}}%
{\XXint\scriptscriptstyle\scriptscriptstyle{#1}}%
\!\int}
\def\XXint#1#2#3{{\setbox0=\hbox{$#1{#2#3}{\int}$ }
\vcenter{\hbox{$#2#3$ }}\kern-.6\wd0}}

\numberwithin{equation}{section}

\newtheorem{theorem}{Theorem}[section]
\newtheorem{proposition}[theorem]{Proposition}
\newtheorem{corollary}[theorem]{Corollary}

\newtheorem{lemma}[theorem]{Lemma}

\newtheorem*{conjecture*}{Conjecture}

\theoremstyle{definition}
\newtheorem{definition}[theorem]{Definition}

\theoremstyle{remark}

\newtheorem{remark}[theorem]{Remark}
 
\newtheorem*{claim*}{Claim} 

\newcommand{\LlogL}{L\!\log\!L}

\newcommand{\ssD}{\slashed{D}}

\title{Weyl's Laws and Connes' Integration Formulas for Matrix-Valued $\LlogL$-Orlicz Potentials}

\author{Rapha\"el Ponge}
 \address{School of Mathematics, Sichuan University, Chengdu, China}
 \email{ponge.math@icloud.com}

\begin{document}
\begin{abstract}
Thanks to the Birman-Schwinger principle, Weyl's laws for Birman-Schwinger operators yields semiclassical Weyl's laws for the corresponding Schr\"odinger operators. In a recent preprint Rozenblum established quite general Weyl's laws for Birman-Schwinger operators associated with pseudodifferential operators of critical order and potentials that are product of $\LlogL$-Orlicz functions and Alfhors-regular measures supported on a submanifold. In this paper, we show that, for matrix-valued $\LlogL$-Orlicz potentials supported on the whole manifold, Rozenblum's results are direct consequences of the Cwikel-type estimates on tori recently established by Sukochev-Zanin. As applications we obtain CLR-type inequalities and semiclassical Weyl's laws for critical Schr\"odinger operators associated with  matrix-valued $\LlogL$-Orlicz potentials. Finally, we explain how the Weyl's laws of this paper imply a strong version of Connes' integration formula for matrix-valued $\LlogL$-Orlicz potentials. 
\end{abstract}

\maketitle 

\section{Introduction}
The study of Schr\"odinger operators $h^2\Delta+V$, and more generally fractional Schr\"odinger operators $h^{n/p}\Delta^{n/2p}+V$ on domains of $\R^n$, as well as on manifolds, is an important focus of interest in mathematical physics. Of special interest  are semiclassical Weyl's laws that describe the semiclassical behaviour  of the number of bound states, i.e., the number of negative eigenvalues $N^{-}(h^2\Delta+V)$. Thanks to the Birman-Schwinger principle this translates into determining the asymptotic distribution of the negative eigenvalues of Birman-Schwinger operators $\Delta^{-n/2p}V\Delta^{-n/2p}$. 
In the 60s and 70s Weyl's laws for positive and negative eigenvalues of Birman-Schwinger operators and semiclassical Weyl's laws for the corresponding Schr\"odinger operators were obtained on $\R^n$ and bounded domains of $\R^n$ for $p<1$ with $V\in L_1$, and for $p>1$ with $V\in L_p$ (see~\cite{BS:JFAA70, Ro:SMD72, Ro:SM76}).  In particular, we have semiclassical Weyl's laws for Schr\"odinger operators $\Delta+V$ with $V\in L_{n/2}$ in any dimension~$n\geq 3$ (see also~\cite{Cw:AM77, Li:BAMS76, Si:TAMS76}). 

The critical case $p=1$, including Schr\"odinger operators in 2D, is the most challenging case. Significant progress on this case was made in \emph{even} dimension by Solomyak~\cite{So:IJM94, So:PLMS95} in the mid-90s.  He obtained eigenvalues estimates and Weyl's laws for operators $\Delta^{-n/4}V\Delta^{-n/4}$ for $\LlogL$-Orlicz potentials on bounded domains of $\R^n$. This yields semiclassical Weyl's laws for critical Schr\"odinger operators $h^2\Delta^{n/2}+V$ in this setting. On $\R^{n}$, once again with $n$ even, Solomyak further obtained eigenvalues inequalities for operators $(\Delta^{-n/4} +\mu^2)V(\Delta^{-n/4} +\mu^2)$, $\mu>0$,  where $V$ is a locally $\LlogL$-Orlicz. Thanks to the Birman-Schwinger principle this provides bounds on the number of bound states~$<-\mu^2$ of the critical Schr\"odinger operator $\Delta^{n/2} +V$. There are also results for $\mu=0$ by suitably restricting the domain of $\Delta$ (see~\cite{So:PLMS95}). 

Solomyak's results triggered an intensive activity on eigenvalues estimates for Schr\"odinger operators in 2D (see, e.g., \cite{BL:CPAM96, CKMW:JMP03, GN:ARMA15, Ka:JMA17, KS:JMP20, KMW:FBS02, LS:JST13, MS:JMS12, St:PAMS04}). We also refer to~\cite{BLS:AFM97} for further results on critical Schr\"odinger operators in even dimension~$\geq 2$. However, the odd dimensional case remained merely unattended until recent articles of Rozenblum-Shargorodsky~\cite{RS:EMS21} and Rozenblum~\cite{Ro:arXiv21}. In particular, by using the eigenvalue inequalities of~\cite{RS:EMS21}, Rozenblum~\cite{Ro:arXiv21} established Weyl's laws in \emph{any} dimension $n\geq 2$ for operators $P^*uP$, where $P$ is a \psido\ of order $-n/2$ on a closed manifold $M$ and $u$ is a potential of the form $u=f\mu$, where $\mu$ is an Alfhors-regular measure supported on a regular submanifold $\Sigma \subset M$ and $f$ is a real-valued $\LlogL$-Orlicz function on $\Sigma$ with respect to $\mu$ (see also~\cite{RT:arXiv21} for the non-critical case).  The main results of~\cite{Ro:arXiv21} further include an extension of Connes' trace theorem for these operators. This shows that even singular measures can be recovered from Connes' integral. This also exhibits an interesting link between semiclassical analysis and Connes' noncommutative geometry.  

Meanwhile, Sukochev-Zanin~\cite{SZ:arXiv20} obtained Cwikel-type estimates on any torus $\T^n$, $n\geq 2$, for operators of the form $(1+\Delta)^{-n/4}f(1+\Delta)^{-n/4}$, where $f$ is an $\LlogL$-Orlicz function on $\T^n$. One goal of this paper is to show that, in the case of smooth measures, the results of Rozenblum~\cite{Ro:arXiv21} are direct consequences of the Cwikel-type estimates of~\cite{SZ:arXiv20}. Although, we don't obtain the results at the level of generality as in~\cite{Ro:arXiv21}, for smooth measures this still allows us to deal with the odd dimensional case, which had been out of focus until recently. Furthermore, our approach simplifies and supersedes the main results of~\cite{SZ:arXiv21}. 

Throughout this paper we work with a closed Riemannian manifold $(M^n,g)$ and a Hermitian vector bundle $E$ over $M$. Thus, the results of this paper are established for matrix-valued potentials, i.e., $\LlogL$-Orlicz sections of $\End(E)$. The main ingredient is a Cwikel-type estimate for operators of the form $QuP$, where $P$ and $Q$ are operators in $\Psi^{-n/2}(M,E)$ and $u$ is a matrix-valued $\LlogL$-Orlicz potential. In this case, the operator $QuP$ is in the weak trace class $\sL_{1,\infty}$, and we have
\begin{equation}
 \big\|QuP\big\|_{1,\infty} \leq C_{PQ}\|u\|_{\LlogL},
 \label{eq:Intro.CwikelPQ}
\end{equation}
where the constant $C_{PQ}$ depends only on $P$ and $Q$ (see Proposition~\ref{prop:Orlicz.Cwikel-ME-PsiDOs}). This extends the ``specific'' Cwikel-type estimates on tori of~\cite{SZ:arXiv20}. In fact, thanks to the (pseudo-)local nature of these estimates it is a routine argument to get the former from the latter (compare~\cite{SZ:arXiv21}). 

Once the Cwikel-type estimates~(\ref{eq:Intro.CwikelPQ}) are  established, there is no difficulty to get Weyl laws for operators $QuP$ as above by combining these estimates with the perturbation theory of Birman-Solomyak~\cite{BS:JFAA70} and the Weyl's laws for negative order \psidos\ of Birman-Solomyak~\cite{BS:VLU77, BS:VLU79, BS:SMJ79} (see Theorem~\ref{thm:Orlicz.WeylQuP} for the precise statement). In particular, if $P$ and $Q$ are operators in $\Psi^{-n/2}(M,E)$ and $u$ is a matrix-valued  $\LlogL$-Orlicz potential, then
\begin{equation}
   \lim_{j\rightarrow \infty} j\mu_j\left(QuP\right)=  \frac1{n} (2\pi)^{-n} \!  \int_{S^*M} \tr_E\big[ \left|\sigma(Q)(x,\xi)u(x) \sigma(P)(x,\xi)\right| \big] dxd\xi, 
   \label{eq:Intro.Weyl-QuP} 
\end{equation}
where $\mu_0(QuP)\geq \mu_1(QuP)\geq \cdots$ are the singular values of $QuP$, and $\sigma(P)$ (resp., $\sigma(Q)$) is the principal symbol of $P$ (resp., $Q$). 
If $Q=P^*$ and $u(x)^*=u(x)$, then the operator $P^*uP$ is selfadjoint and satisfies the Weyl's law, 
\begin{equation}
   \lim_{j\rightarrow \infty} j\lambda^\pm_j\left(P^*uP\right)=  \frac1{n} (2\pi)^{-n} \!  \int_{S^*M} \tr_E\left[ \big( \sigma(P)(x,\xi)^*u(x) \sigma(P)(x,\xi)\big)_\pm \right] dxd\xi, 
   \label{eq:Intro.Weyl-PuP} 
\end{equation}
 where $\lambda_0^\pm(P^*uP)\geq  \lambda_1^\pm(P^*uP)\geq \cdots \geq 0$ are the positive/negative eigenvalues of $P^*uP$ and the subscript $\pm$ refers to the positive/negative parts in $\End(E_x)$. If $\Delta_E=\nabla^*\nabla$ is the Laplacian of some Hermitian connection on $E$ and we take $P=Q=\Delta_E^{-n/4}$, then we get
\begin{gather}
\lim_{j\rightarrow \infty} j\mu_j\left(\Delta_E^{-\frac{n}{4}}u\Delta_E^{-\frac{n}{4}}\right)= \frac1{n} (2\pi)^{-n} \! \int_M  \tr_E\big[|u(x)|\big] \sqrt{g(x)}dx,
\label{eq:Intro.Weyl-|DuD|} \\ 
\lim_{j\rightarrow \infty} j\lambda^\pm_j\left(\Delta_E^{-\frac{n}{4}}u\Delta_E^{-\frac{n}{4}}\right)= \frac1{n} (2\pi)^{-n} \! \int_M  \tr_E\big[u(x)_\pm\big] \sqrt{g(x)}dx \qquad  \text{if $u(x)^*=u(x)$}.
\label{eq:Intro.Weyl-DuD} 
\end{gather}

In the form stated above the Weyl's law~(\ref{eq:Intro.Weyl-QuP}) is new. As mentioned above, in the scalar case Rozenblum~\cite{Ro:arXiv21} estasblished the Weyl's law~(\ref{eq:Intro.Weyl-PuP}) for a larger class of potentials. 
Still in the scalar case, Sukochev-Zanin~\cite{SZ:arXiv21} obtained the estimates~(\ref{eq:Intro.CwikelPQ}) and the Weyl's law~(\ref{eq:Intro.Weyl-PuP}) in the special case $P=Q=(1+\Delta_g)^{-n/4}$. 
The  approach in~\cite{SZ:arXiv21} uses the Cwikel-type estimates of~\cite{SZ:arXiv20}, but it also relies on some deep results on commutators in weak Schatten classes from~\cite{DFWW:AIM04, HSZ:Preprint}. Once again, our prime goal here is explaining that, for our class of $\LlogL$-Orlicz potentials, the Cwikel-type estimates~(\ref{eq:Intro.CwikelPQ}) and all the Weyl's laws~(\ref{eq:Intro.Weyl-QuP})--(\ref{eq:Intro.Weyl-DuD}) are direct consequences of the Cwikel-type estimates of~\cite{SZ:arXiv20}. We also stress that the Weyl's laws~(\ref{eq:Intro.Weyl-|DuD|})--(\ref{eq:Intro.Weyl-DuD}) hold \emph{verbatim} if we replace $\Delta_E$ by any Laplace-type operator. In particular, they hold for squares of Dirac-type operators (see Proposition~\ref{prop:Orlicz.Dirac}). 
 
We also give applications of the above to the semiclassical analysis of critical Schr\"odinger operators $\Delta_E^{n/2}+V$, where $V$ is a (Hermitian) matrix-valued $\LlogL$-Orlicz potential. The Cwikel estimate~(\ref{eq:Intro.CwikelPQ}) for $Q=P=\Delta_E^{-n/2}$ ensures that $V$ is $\Delta_E^{n/2}$-form compact, and so the operator $\Delta_E^{n/2}+V$ makes sense as a form sum. It is selfadjoint and bounded from below and has pure discrete spectrum. As above we denote by $N^{-}( \Delta_E^{n/2}+V)$ its number of negative eigenvalues.

 By a standard application of the abstract Birman-Schwinger principle mentioned above, the Weyl's law~(\ref{eq:Intro.Weyl-DuD}) implies the following semiclassical Weyl's law (\emph{cf}.~Corollary~\ref{cor:SC.SC-Weyl}), 
\begin{equation}
 \lim_{h \rightarrow 0^+} h^nN^{-}\big(h^n \Delta_E^{\frac{n}{2}}+V\big)= \frac1{n} (2\pi)^{-n} \! \int_M  \tr_E\big[V(x)_{-}\big] \sqrt{g(x)}dx.
 \label{eq:Intro.SC-Weyl} 
\end{equation}
Semiclassical Weyl's laws are not considered in~\cite{Ro:arXiv21}. However, the same argument as above shows that the Weyl's laws of~\cite{Ro:arXiv21} similarly yields a semiclassical Weyl's law for the type of potentials considered in~\cite{Ro:arXiv21}.

Another application concerns a version of the Cwikel-Lieb-Rozenblum (CLR) inequality for $\LlogL$-Orlicz potentials (see Section~\ref{sec:SC} for background on the CLR inequality). By combining the Cwikel-type estimates~(\ref{eq:Intro.CwikelPQ}) for $P=Q=\Delta_E^{-n/4}$ with the borderline Birman-Schwinger principle of~\cite{MP:Part1} we immediately get the following CLR-type inequality, 
\begin{equation}
 N^{-}\big( \Delta_E^{\frac{n}{2}}+V\big) - N^+(\Pi_0V_{-}\Pi_0) \leq C_{\nabla}\|V_{-}\|_{\LlogL},
 \label{eq:Intro.CLR}   
\end{equation}
 where the constant $C_\nabla$ does not depend on $V$  (see Corollary~\ref{cor:SC.CLR}). Here $\Pi_0$ is the orthogonal projection onto $\ker \Delta_E$ and $N^+(\Pi_0V_{-}\Pi_0)$ is the number of positive eigenvalues of $\Pi_0V_{-}\Pi_0$. In particular, we obtain a version of the CLR inequality for Schr\"odinger operators $\Delta_E+V$ on 2-dimensional closed manifolds. Note that the original CLR inequality does not hold on $\R^2$. 
 
We also clarify the links between the Weyl's laws~(\ref{eq:Intro.Weyl-QuP})--(\ref{eq:Intro.Weyl-DuD}) and Connes' integration. In the framework of Connes' noncommutative geometry~\cite{Co:NCG} the role of the integral is played by positive traces on the weak trace class $\sL_{1,\infty}$ (see also Section~\ref{sec:NCG} and the references therein). We define the NC integral of an operator $A\in \sL_{1,\infty}$ by
\begin{equation}
 \bint A:= \lim_{N\rightarrow \infty} \frac{1}{\log N} \sum_{j<N} \lambda_j(A) \quad \text{provided the limit exists}, 
 \label{eq:Intro.NC-Integral} 
\end{equation}
where $(\lambda_j(A))_{j\geq 0}$ is any eigenvalue sequence for $A$ such that $|\lambda_0(A)|\geq |\lambda_1(A)|\geq \cdots$. An operator for which the above limit exists is called \emph{measurable}. The measurable operators form a closed subspace of $\sL_{1,\infty}$ which contains the commutator subspace $\Com(\sL_{1,\infty})$ on which the NC integral is a positive linear trace. 

Measurability of weak trace-class operators is equivalently described in terms of Dixmier traces (see~\cite{Di:CRAS66, Co:NCG, LSZ:Book, Po:NCIntegration}). These traces extends to the Dixmier-Macaev class $\fM_{1,\infty}$, which is strictly larger than $\sL_{1,\infty}$. There are many traces on $\sL_{1,\infty}$ that are not Dixmier or do not extend to $\fM_{1,\infty}$. We say that an operator $A\in \sL_{1,\infty}$ is \emph{strongly measurable} if, for every (normalized) positive trace $\varphi$ on $\sL_{1,\infty}$, the value $\varphi(A)$ is given by the limit in~(\ref{eq:Intro.NC-Integral}). We stress that this notion of measurability is stronger than the notion of measurability considered in~\cite{Ro:arXiv21}, which is defined in terms of traces on $\fM_{1,\infty}$. Moreover, as pointed out in~\cite{Po:NCIntegration}, Weyl's laws for operators in $\sL_{1,\infty}$ imply strong measurability in the sense meant in this paper. 

Let $(M^n,g)$ be a closed Riemannian manifold and $E$ a Hermitian vector bundle over $M$. By Connes's trace theorem~\cite{Co:CMP88, KLPS:AIM13} any operator $P\in \Psi^{-n}(M,E)$ is strongly measurable, and we have 
\begin{equation*}
 \bint P = \frac1{n} \int_{S^*M} \tr_E\big[\sigma(P)(x,\xi)\big]dxd\xi. 
\end{equation*}
In particular, for $P=\Delta_E^{-n/4}f\Delta_E^{-n/4}$ with $f\in C^\infty(M)$ we get Connes' integration formula, 
\begin{equation}
 \bint \Delta_E^{-\frac{n}{4}}u\Delta_E^{-\frac{n}{4}} = \frac1{n} (2\pi)^{-n} \! \int_M  \tr_E\big[u(x)\big] \sqrt{g(x)}dx.
 \label{eq:Intro.Connes-integration} 
\end{equation}
This shows that Connes' integral recaptures the Riemannian volume measure. 

The Weyl's laws~(\ref{eq:Intro.Weyl-QuP})--(\ref{eq:Intro.Weyl-PuP}) and the observations of~\cite{Po:NCIntegration} immediately imply the following extension of Connes' trace theorem (see Theorem~\ref{thm:Orlicz-trace-formula}). If  $P$ and $Q$ are operators in $\Psi^{-n/2}(M,E)$ and $u$ is a matrix-valued  potential $\LlogL$-Orlicz potential, then the operators $QuP$ and $|QuP|$ are both strongly measurable, and we have 
\begin{gather}
 \bint QuP = \frac1{n} (2\pi)^{-n} \!  \int_{S^*M} \tr_E \big[ \sigma(Q)(x,\xi)u(x) \sigma(P)(x,\xi)\big]dxd\xi,
 \label{eq:Intro.trace-formula-QuP}\\
\bint \big|QuP\big| =
 \frac1{n} (2\pi)^{-n} \!  \int_{S^*M} \tr_E \big[\left|\sigma(Q)(x,\xi)u(x) \sigma(P)(x,\xi)\right|\big]dxd\xi.
 \label{eq:Intro.trace-formula-|QuP|} 
\end{gather}
In particular, Connes' integration formula~(\ref{eq:Intro.Connes-integration}) holds \emph{verbatim} for matrix-valued $\LlogL$-potentials. This also leads to semiclassical interpretation of Connes' integration formula (see Remark~\ref{rmk:Int-Orlicz.SC-Connes}). 

Note that~(\ref{eq:Intro.trace-formula-|QuP|}) yields an integration formula for the absolute value $|\Delta_E^{-\frac{n}{4}}u\Delta_E^{-\frac{n}{4}}|$ as well. Namely, if $u$ is any matrix-valued $\LlogL$-Orlicz potential, then 
\begin{equation}
 \bint \bigg|\Delta_E^{-\frac{n}{4}}u\Delta_E^{-\frac{n}{4}}\bigg| = \frac1{n} (2\pi)^{-n} \! \int_M  \tr_E\big[|u(x)|\big] \sqrt{g(x)}dx.
 \label{eq:Intro.Connes-integration-|DuD|} 
\end{equation}
We also have versions of the trace formulas~(\ref{eq:Intro.Connes-integration}) and~(\ref{eq:Intro.Connes-integration-|DuD|}) for (squares of) Dirac-type operators and matrix-valued $\LlogL$-Orlicz potentials (see Proposition~\ref{prop:Orlicz.Dirac}). 

The trace formula~(\ref{eq:Intro.trace-formula-|QuP|}) is new. In the scalar case Rozenblum~\cite{Ro:arXiv21} established the trace formula~(\ref{eq:Intro.trace-formula-QuP}) for a larger class of potentials. As mentioned above, a weaker notion of measurability is used in~\cite{Ro:arXiv21}. However, the same arguments as above show that the operators under consideration in~\cite{Ro:arXiv21} are strongly measurable in the sense used in this paper. 
Still in the scalar case, a weaker version of the trace formula~(\ref{eq:Intro.trace-formula-QuP}) in the special case $P=Q=(1+\Delta)^{-n/4}$ and $u\in L_\infty(M)$,  is given in~\cite{SZ:arXiv21}.

The remainder of this paper is organized as follows. In Section~\ref{sec:NCInt}, we review some facts regarding weak Schatten classes and Birman-Solomyak's results on Weyl's laws for compact operators. In Section~\ref{sec.Orlicz}, we establish the Cwikel-type estimates~(\ref{eq:Intro.CwikelPQ}). 
In Section~\ref{sec:Weyl-Orlicz}, we derive the Weyl's laws~(\ref{eq:Intro.Weyl-QuP})--(\ref{eq:Intro.Weyl-DuD}) for matrix-valued $\LlogL$-Orlicz potentials. We also establish the semiclassical Weyl's law~(\ref{eq:Intro.SC-Weyl}) and the CLR-type inequality~(\ref{eq:Intro.CLR}). In Section~\ref{sec:NCG}, we review the main facts regarding Connes' integration and its relationship with Weyl's laws. Finally, in Section~\ref{sec:NCInt-Orlicz}, we deal with the integration formulas~(\ref{eq:Intro.trace-formula-QuP})--(\ref{eq:Intro.trace-formula-|QuP|}) and their consequences.

\subsection*{Aknowledgements} I wish to thank Grigori Rozenblum, Edward McDonald, and Fedor Sukochev for various stimulating discussions related to the subject matter of this article. I also thank the University of Ottawa for its hospitality during the whole preparation of this paper.

\section{Weak Schatten Classes and Weyl's Laws for Compact Operators}\label{sec:NCInt} 
In this section, we recall the main facts regarding weak Schatten classes and Birman-Solomyak's results on Weyl's laws for compact operators. 

Throughout this section we let $\sH$ be a (separable) Hilbert space with inner product $\scal{\cdot}{\cdot}$. We denote by $\sL(\sH)$ the $C^*$-algebra of bounded operators on $\sH$ with norm $\|\cdot\|$. 

\subsection{Weak Schatten Classes} \label{subsec:schatten}
We briefly review the main definitions and properties regarding weak Schatten classes on $\sH$. We refer to~\cite{Si:AMS05, GK:AMS69} for further details. 

Let $\sK$ be (closed) the ideal of compact operators on $\sH$. Given any operator $T\in \sK$ we denote by $(\mu_j(T))_{j\geq 0}$ its sequence of \emph{singular values}, i.e., $\mu_j(T)$ is the $(j+1)$-th eigenvalue counted with multiplicity of the absolute value $|T|=\sqrt{T^*T}$. The \emph{min-max principle} states that
\begin{align}
 \mu_j(T)&=\min \left\{\|T_{|E^\perp}\|;\ \dim E=j\right\}.
 \label{eq:min-max} 
\end{align}
The min-max principle implies the following properties (see, e.g., \cite{GK:AMS69, Si:AMS05}), 
\begin{gather}
 \mu_j(T)=\mu_j(T^*)=\mu_j(|T|),
 \label{eq:Quantized.properties-mun1}\\
 \mu_{j+k}(S+T)\leq \mu_j(S) + \mu_k(T),
 \label{eq:Quantized.properties-mun2}\\
 \mu_j(ATB)\leq \|A\| \mu_j(T) \|B\|, \qquad A, B\in \sL(\sH). 
 \label{eq:Quantized.properties-mun3} 
\end{gather}
The inequality~(\ref{eq:Quantized.properties-mun2}) is known as Ky Fan's inequality~\cite{Fa:PNAS51}. 

For $p\in (0,\infty)$, the weak Schatten class  $\sL_{p,\infty}$ is defined by
\begin{equation*}
 \sL_{p,\infty}:=\left\{T\in \sK; \ \mu_j(T)=\op{O}\big(j^{-\frac1p}\big)\right\}. 
\end{equation*}
 This is a two-sided ideal and a quasi-Banach ideal with respect to the quasi-norm,
\begin{equation}\label{def:lp_infty_quasinorm}
 \|T\|_{p,\infty}:=\sup_{j\geq 0}\;(j+1)^{\frac{1}{p}}\mu_j(T), \qquad T\in \sL_{p,\infty}. 
\end{equation}
For $p> 1$, the quasi-norm is equivalent to the norm,
\begin{equation*}
 \|T\|_{p,\infty}':=  \sup_{N\geq 1} N^{-1+\frac1{p}}\sum_{j<N} \mu_j(T), \qquad T\in \sL_{p,\infty}.
\end{equation*}
Thus, in this case $\sL_{p,\infty}$ is a Banach ideal with respect to that equivalent norm. 

In addition, we denote by $(\sL_{p,\infty})_{0}$ the closure in $\sL_{p,\infty}$ of the finite-rank operators. We have
\begin{equation*}
 \big(\sL_{p,\infty}\big)_{0}=\left\{T\in \sK; \ \mu_j(T)=\op{o}\big(j^{-\frac1p}\big)\right\}.
\end{equation*}
We note the continuous inclusions, 
\begin{equation*}
 \sL_p \subsetneq  \big(\sL_{p,\infty}\big)_{0}\subsetneq \sL_{p,\infty}  \subsetneq \sL_{q}, \qquad 0<p<q.
\end{equation*}

Let $\sH'$ be another Hilbert space, and let $\iota: \sH'\rightarrow \sH$ be a Hilbert space embedding, i.e., a continuous linear map which is one-to-one and has closed range. For instance, we may take $\iota$ to be an isometry. Set $ \sH_1=\ran \iota$; by assumption this is a closed subspace of $\sH$. We get a continuous linear isomorphism  $\iota:\sH'\rightarrow \sH_1$ with inverse   $\iota^{-1}:\sH_1\rightarrow \sH'$. Thus, if $\pi:\sH\rightarrow \sH_1$ is the orthogonal projection onto $\sH_1$, then $\iota^{-1}\circ \pi$ is a left-inverse of $\iota$. 
The pushforward $\iota_*:\sL(\sH')\rightarrow \sL(\sH)$ is then defined by
\begin{equation}
 \iota_*A= \iota \circ A \circ (\iota^{-1}\circ \pi), \qquad A \in \sL(\sH'). 
 \label{eq:App.embedding-sLsH}
\end{equation}
This is an embedding of (unital) Banach algebras.  If $\iota$ is an isometry, then $\iota_*:\sL(\sH')\rightarrow \sL(\sH)$ is an isometric $*$-homomorphism. 

In what follows, for $p>0$, we denote by $\sL_{p,\infty}(\sH)$ and $\sL_{p,\infty}(\sH')$ the corresponding weak Schatten classes on $\sH$ and $\sH'$, respectively. 

\begin{proposition}[see, e.g., {\cite[Appendix]{Po:NCIntegration}}]\label{prop:App.eigenvalues} The following holds. 
\begin{enumerate}
 \item If $A$ is a compact operator on $\sH'$, then $A$ and $\iota_*A$ have the same non-zero eigenvalues with the same algebraic multiplicities.
 
 \item There is $c>0$ such that, for every compact operator $A$ on $\sH'$, we have 
 \begin{equation*}
 c^{-1} \mu_j(A)\leq \mu_j\big(\iota_*A) \leq c \mu_j(A) \qquad \forall j\geq 0. 
\end{equation*}
We may take $c=1$ when $\iota$ is an isometric embedding. 

\item  The pushforward map~(\ref{eq:App.embedding-sLsH}) induces a quasi-Banach embedding $ \iota_*: \sL_{p,\infty}(\sH')\rightarrow \sL_{p,\infty}(\sH)$ for each $p>0$. This is an isometric embedding when $\iota$ is an isometry. 
\end{enumerate}
\end{proposition}

\subsection{Weyl operators} 
In what follows, given any selfadjoint operator $A\in \sK$, we denote by $(\pm\lambda^\pm(A))_{j\geq 0}$ its sequences of positive and negative eigenvalues, so that
\begin{equation*}
  \lambda^{\pm}_0(A)\geq  \lambda^{\pm}_1(A) \geq \cdots \geq 0,
\end{equation*}
 where each eigenvalue is repeated according to multiplicity.  More precisely, $\lambda_j^\pm(A)=\mu_j(A^{\pm})$, where $A^\pm=\frac12(|A|\pm A)$ are the positive and negative parts of $A$. 
 
 We refer to~\cite[\S9.2]{BS:Book} for the main properties of the positive/negative eigenvalue sequences of selfadjoint compact operators. In particular, we have the following min-max principle (\emph{cf}.~\cite[Theorem 9.2.4]{BS:Book}), 
\begin{equation*}
 \lambda^\pm_j(A)= \min \bigg\{ \max_{0\neq \xi\in E^\perp} \pm \frac{\scal{A\xi}{\xi}}{\scal\xi\xi}; \  \dim E=j\bigg\}.
\end{equation*}
This implies the following version of Ky Fan's inequality (\emph{cf}.~\cite[Theorem 9.2.8]{BS:Book}), 
\begin{equation}
 \lambda^\pm_{j+k}(A+B)\leq \lambda^\pm_j(A) + \lambda^\pm_k(B), \qquad j,k\geq 0. 
 \label{eq:Bir-Sol.Weyl-ineq-lambdapm}
\end{equation}

\begin{definition}
We say that $A\in \sL_{p,\infty}$, $p>0$ is a \emph{Weyl operator} if one of the following conditions applies:
\begin{enumerate}
 \item[(i)] $A\geq 0$ and $\lim j^{1/p}\lambda_j(A)$ exists. 
 
 \item[(ii)] $A^*=A$ and $\lim j^{1/p}\lambda_j^+(A)$ and $\lim j^{1/p}\lambda_j^-(A)$ both exist. 
 
 \item[(iii)] The real part $\Re A=\frac12(A+A^*)$ and the imaginary part $\Im A= \frac1{2i}(A-A^*)$ of are both Weyl operators in the sense of (ii). 
\end{enumerate}
\end{definition}

We denote by $\sW_{p,\infty}$ the class of Weyl operators in $\sL_{p,\infty}$.  If $A\in \sW_{p,\infty}$, $A\geq 0$, we set 
\begin{equation*}
 \Lambda(A)=\lim_{j\rightarrow \infty} j^{\frac1p} \lambda_j(A). 
\end{equation*}
If $A=A^*\in \sW_{p,\infty}$, we set 
\begin{equation*}
  \Lambda^\pm(A)=\lim_{j\rightarrow \infty} j^{\frac1p} \lambda^\pm_j(A). 
\end{equation*}
Finally, for a general $A\in \sW_{p,\infty}$, we define 
\begin{equation*}
  \Lambda^\pm(A)=  \Lambda^\pm\big(\Re A\big) + i \Lambda^\pm\big(\Im A\big). 
\end{equation*}

\begin{remark}\label{rmk:Bir-Sol.counting}
 Given any selfadjoint compact operator $A$ on $\sH$, its counting functions are given by
\begin{equation*}
 N^\pm(A;\lambda):=\# \big\{j; \ \lambda_j^\pm(A)>\lambda\big\}, \qquad \lambda>0. 
\end{equation*}
If $A\in \sW_{p,\infty}$, $p>0$, then (see, e.g., \cite[Proposition~13.1]{Sh:Springer01}), we have
\begin{equation}
\lim_{\lambda \rightarrow 0^+} \lambda^pN^\pm(A;\lambda)= \lim_{j\rightarrow \infty} j \lambda_j^\pm(A)^p= \Lambda^\pm(A)^p. 
\label{eq:Bir-Sol.counting-Lambda}
\end{equation}
\end{remark}

It is also convenient to introduce the following class of operators. 

\begin{definition}
 $\sW_{|p,\infty|}$, $p>0$, consists of operators $A\in \sL_{p,\infty}$ such that $|A|\in \sW_{p,\infty}$, i.e., $\lim j^{1/p}\mu_j(A)$ exists. 
\end{definition}

In particular, if $A\in \sW_{|p,\infty|}$, then 
\begin{equation*}
 \Lambda\big(|A|\big)= \lim_{j\rightarrow \infty} j^{\frac1p}\mu_j(A). 
\end{equation*}

The perturbation theory of Birman-Solomyak~\cite[Theorem~4.1 \& Remark 4.2]{BS:JFAA70} implies the following results (see also~\cite{Po:NCIntegration}). 

\begin{proposition}\label{prop:Bir-Sol.closedness} 
The following holds. 
\begin{enumerate}
 \item $\sW_{p,\infty}$ is a closed subset of $\sL_{p,\infty}$ on which $\Lambda^\pm:\sW_{p,\infty}\rightarrow \C$ are continuous maps. 
 
 \item Let $A\in \sW_{p,\infty}$ and $B\in (\sL_{p,\infty})_0$. Then $A+B\in \sW_{p,\infty}$, and we have 
\begin{equation*}
  \Lambda^\pm(A+B)=\Lambda^\pm(A). 
\end{equation*}
\end{enumerate}
\end{proposition}

\begin{proposition}
\label{prop:Bir-Sol.closedness-sing}
The following holds. 
\begin{enumerate}
 \item $ \sW_{|p,\infty|}$ is a closed subset of $\sL_{p,\infty}$ on which the functional $A \rightarrow \Lambda(|A|)$ is continuous. 

\item  If $A\in \sW_{|p,\infty|}$ and $B\in (\sL_{p,\infty})_0$. Then $A+B\in \sW_{|p,\infty|}$, and we have 
\begin{equation*}
  \Lambda^\pm\big(|A+B|\big)=\Lambda^\pm\big(|A|\big). 
\end{equation*}
\end{enumerate}
\end{proposition}

\begin{remark}
 The 2nd part of Proposition~\ref{prop:Bir-Sol.closedness-sing} is due originally to Ky Fan~\cite[Theorem 3]{Fa:PNAS51} (see also~\cite[Theorem II.2.3]{GK:AMS69}). 
 \end{remark}

\begin{remark}
 We stress out that Proposition~\ref{prop:Bir-Sol.closedness} and Proposition~\ref{prop:Bir-Sol.closedness-sing} are elementary consequences of the Ky Fan's inequalities~(\ref{eq:Quantized.properties-mun2}) and~(\ref{eq:Bir-Sol.Weyl-ineq-lambdapm}) 
 (see~\cite{BS:JFAA70}; compare~\cite{SZ:arXiv21}). 
\end{remark}

\subsection{Example: Weyl's laws for negative order \psidos} \label{sec:Weyl-neg-PsiDOs}\label{sec:Bir-Sol-PDOs} 
Suppose that $(M^n,g)$ is a closed Riemannian manifold and $E$ a Hermitian vector bundle over $M$.  Given any $m\in \R$ we denote by $\Psi^m(M,E)$ the space of $m$-th order classical pseudodifferential operators (\psidos) $P:C^\infty(M,E)\rightarrow C^\infty(M,E)$. For $P\in \Psi^m(M,E)$ we denote by $\sigma(P)(x,\xi)$ its principal symbol; this is a smooth section of $\End(E)$ over $T^*M\setminus 0$. Recall that any $P\in \Psi^m(M,E)$ with $m\leq 0$ extends to a bounded operator $P:L_2(M,E)\rightarrow L_2(M,E)$; if $m<0$, then this operator is actually in the weak Schatten class $\sL_{p,\infty}$ with $p=n|m|^{-1}$.

\begin{proposition}[Birman-Solomyak~\cite{BS:VLU77, BS:VLU79, BS:SMJ79}]\label{prop:Bir-Sol-asymp} Let $P\in \Psi^{-m}(M,E)$, $m<0$, and set $p=nm^{-1}$. 
\begin{enumerate}
 \item $P$ and $|P|$ are Weyl operators in $\sL_{p,\infty}$. 
 
 \item  We have
\begin{equation}
\lim_{j\rightarrow \infty} j^{\frac1p} \mu_j(P)=\bigg[ \frac1{n} (2\pi)^{-n} \int_{S^*M} \tr_E\big[ |\sigma(P)(x,\xi)|^{p} \big] dx d\xi\bigg]^{\frac1{p}}.
 \label{eq:Weyl.Bir-Sol-mu} 
\end{equation}

 \item If $P$ is selfadjoint, then 
\begin{equation}
 \lim_{j\rightarrow \infty} j^{\frac1p} \lambda^\pm_j(P)= \bigg[\frac1{n} (2\pi)^{-n} \int_{S^*M} \tr_E\big[ \sigma(P)(x,\xi)_\pm^{p} \big] dx d\xi\bigg]^{\frac1{p}}. 
 \label{eq:Weyl.Bir-Sol-selfadjoint}
\end{equation}
\end{enumerate}
 \end{proposition}

\begin{remark}
In~\cite{BS:VLU77, BS:VLU79} Birman-Solomyak established the above Weyl's laws for compactly supported pseudodifferential operators on $\R^n$ under very low regularity assumptions on the symbols. Furthermore, the symbols are allowed to be anisotropic. This was extended to classical \psidos\ on closed manifolds in~\cite{BS:SMJ79}. Unfortunately, the key technical details are somehwat compressed and contained in the Russian paper~\cite{BS:VLU79}, the translation of which is unavailable. We refer to~\cite{Po:NCIntegration} for a soft proof of Proposition~\ref{prop:Bir-Sol-asymp}.
\end{remark}

\begin{remark}
 We refer to~\cite{AA:FAA96, An:MUSSRS90, BY:JSM84, DR:LNM87, Gr:CPDE14, Iv:Springer19, Po:NCIntegration, Ro:arXiv21, RS:EMS21}, and the references therein, for various generalizations and applications of Birman-Solomyak's asymptotics.  
\end{remark}

\section{Cwikel-Type Estimates for Matrix-Valued $\LlogL$-Orlicz Potentials}\label{sec.Orlicz}  
In this section, we establish general $\sL_{1,\infty}$-Cwikel type estimates for operators of the form $QuP$, where $u$ is a matrix-valued $\LlogL$-Orlicz potential and $Q$ and $P$ are \psidos\ of order $-n/2$ on an $n$-dimensional closed manifold. In the scalar case Rozenblum~\cite{Ro:arXiv21} established such estimates
for potentials of the form $V=f\mu$, where $\mu$ is an Alfhors-regular measure supported on a  regular submanifold $\Sigma \subset M$ and  $f$ is a real-valued $\LlogL$-Orlicz function on $\Sigma$ with respect to $\mu$. Our aim is to explain how in the case of matrix-valued $\LlogL$-Orlicz potentials defined on all $M$ the estimates are easy consequences of the specific Cwikel-type estimates on tori in the recent preprint of Sukochev-Zanin~\cite{SZ:arXiv20}.  

Throughout this section we make the convention that $C_{abd}$ are positive constants which depend only on the parameters $a$, $b$, $d$, etc., and may change from line to line. 

\subsection{Sobolev spaces} 
Given $s\geq 0$,  the (Bessel potential) Sobolev space $W^{s}_2(\R^n)$ consists of all $u\in L_2(\R^n)$ such that $(1+\Delta)^{s/2}u\in L_2(\R^n)$, where $\Delta=-(\partial_{x_1}^2+\cdots +\partial_{x_n}^2)$ is the (positive) Laplacian. This is a  Hilbert space with respect to the norm, 
\begin{equation*}
 \|u\|_{W^{s}_2} =\big\|(1+\Delta)^{\frac{s}{2}}u\big\|_{L_2} = (2\pi)^{-\frac{n}2}\bigg( \int \big(1+|\xi|\big)^s |\hat{u}(\xi)|^2d\xi\bigg)^{\frac12}, \qquad u\in W^{s}_2(\R^n). 
\end{equation*}
Here $\hat{u}(\xi)=\int e^{-ix\cdot \xi}u(x)dx$ is the Fourier transform of $u$. Recall that by Sobolev's embedding theorem, for $s<n/2$ we have a continuous embedding,
\begin{equation}
 W^{s}_2(\R^n)\subset L_p(\R^n), \qquad p=\frac{n}{2}-s. 
 \label{eq:Orlicz.Sobolev-embed-Lp}  
\end{equation}
For $s>n/2$ we have continuous embeddings into H\"older spaces $C^{k,\alpha}(\R^n)$. Namely,  
\begin{equation}
 W^{s}_2(\R^n)\subset C^{k,\alpha}(\R^n),
 \label{eq:Orlicz.Sobolev-embed-Ckalpha}  
\end{equation}
 where $k\in \N_0$ and $\alpha \in (0,1]$ are such that $s-n/2=k+\alpha$. 
 
If $s>0$, then $W^{-s}_2(\R^n)$ consists of all tempered distributions $u\in \sS'(\R^n)$ such that  $(1+\Delta)^{-s/2}u\in L_2(\R^n)$, i.e.,  $(1+|\xi|^2)^{-s/2}\hat{u}\in L_2(\R^n)$. This is also a Hilbert space, which is naturally identified with the anti-linear dual of $W^{s}_2(\R^n)$, i.e., the Hilbert space of continuous anti-linear forms on $W^{s}_2(\R^n)$. More precisely, the inner product $L_2(\R^n)\times L_2(\R^n)\rightarrow \C$ uniquely extends to a nondegenerate continuous sesquilinear pairing $W^{-s}_2(\R^n)\times W^{s}_2(\R^n)\rightarrow \C$. 

Let $(M^n,E)$ be a closed manifold and $E^r$ a Hermitian vector bundle over $M$. For $s\geq 0$ the Sobolev space $W^{s}_2(\R^n)$ consists of all sections $u\in L_2(M,E)$ such that, for every local chart $\kappa:U\rightarrow V\subset \R^n$ over which there is a trivialization $\tau:E_{|U}\rightarrow U\times \C^r$ and for every test function $\varphi\in C^\infty_c(U)$, the pushforward $\kappa_*\tau_*(\varphi u)$ is in $W^{s}_2(\R^n)\otimes \C^r$. 

The Sobolev space $W^{s}_2(M,E)$ is a Hilbert space. For instance, if $(\varphi_i)_{1\leq i \leq N}$ is a finite smooth partition of unity subordinate to an open cover $(U_i)_{1\leq i \leq N}$ such that each open $U_i$ is the domain of chart $\kappa_i:U_i\rightarrow \R^n$ over which there is a trivialization $\tau_i:E_{|U_i}\rightarrow U_i\times \C^r$, then a Hilbert norm is given by
\begin{equation}
 \|u\|_{W_2^s}:= \bigg(\sum_{1\leq i \leq N} \big\|(\kappa_i\circ\tau_i)_*(\varphi_i u)\big\|_{W_2^s}^2\bigg)^{\frac12}, \qquad u\in W_2^s(M,E).
 \label{eq:Orlicz.Sobolev-norm-ME}
\end{equation}
The corresponding locally convex topology does not depend on the choice of the partition of unity. 

We define $W^{-s}_2(M,E)$ to be the anti-linear dual of $W^{s}_2(M,E)$. Recall that if $P\in \Psi^{m}(M,E)$, $m\in \R$, then $P$ extends to a bounded map $P:W^{2,s+m}(M,E)\rightarrow W^{s}_2(M,E)$ for every $s\in \R$. In particular, if $\Delta_E=\nabla^*\nabla$ is the Laplacian of some Hermitian connection on $E$, then we have 
\begin{equation*}
 W^{s}_2(M,E)= \big\{u\in \sD'(M,E); \ (1+\Delta_E)^{s/2}u\in L_2(M,E)\big\}, \qquad s\in \R. 
\end{equation*}
An equivalent Hilbert norm on $W^s_2(M,E)$ then is $u\rightarrow \|(1+\Delta_E)^{s/2}u\|_{L_2}$. 

\subsection{Orlicz spaces and critical Sobolev embedding} Let $(\Omega, \mu)$ be a $\sigma$-finite measure space. For our purpose we may take $\Omega$ to  be $\R^n$ or a bounded domain in $\R^n$ equipped with the Lebesgue measure,  or a closed Riemannian manifold equipped with its Riemannian measure. 

A \emph{Young function} is a convex function $F:[0,\infty)\rightarrow [0,\infty)$ such that
\begin{equation*}
 \lim_{x\rightarrow 0}x^{-1}F(x)=0 \quad \text{and} \quad \lim_{x\rightarrow \infty} x^{-1}F(x)=\infty.
\end{equation*}
The Orlicz space $L_F(\Omega)$ associated to such a function consists of all (classes of) measurable functions $f:\Omega\rightarrow \C$ for which $F(\lambda^{-1}|f|)\in L_1(\Omega)$ for some $\lambda>0$. We refer to~\cite[Chapter 2]{Si:Cambridge11} for a concise and up to date exposition of Orlicz spaces. 
In particular, $L_F(\Omega)$ is a Banach space with respect to the norm, 
\begin{equation}
 \|f\|_{L_F} = \inf\left\{\lambda>0; \ \|F(\lambda^{-1}|f|)\|_{L_1}\leq 1\right\}, \qquad f\in  L_F(\Omega).
 \label{eq:Orlicz.Orlicz-norm} 
\end{equation}

For $F(t)=t^p$, $p\geq 1$, we recover the $L_p$-space $L_p(\Omega)$. For $F(t)=(1+t)\log(1+t)-t$ we get Zygmund's space $\LlogL(\Omega)$. We observe that if $\mu(\Omega)<\infty$, then we have continuous inclusions, 
\begin{equation*}
 L_p(\Omega) \subset \LlogL(\Omega) \subset L_1(\Omega), \qquad p>1.
\end{equation*}
Moreover, in this case $L_\infty(\Omega)$ is dense in $\LlogL(\Omega)$. If we further assume that  $\Omega$ is a bounded domain in $\R^n$ or a compact Riemannian manifold, then $C^\infty_c(\Omega)$ is a dense subspace of $\LlogL(\Omega)$. 

For $q=1,2$ we denote by $\exp(L_q)(\Omega)$ the Orlicz space associated to $F(t)=\exp(t^q)-1-t^q$. If $\Omega=\R^n$ the
relevance of $\exp(L_2)(\R^n)$ stems from filling the gap in Sobolev's embedding theorems~(\ref{eq:Orlicz.Sobolev-embed-Lp})-(\ref{eq:Orlicz.Sobolev-embed-Ckalpha}). More precisely, we have the following version of M\"oser-Trudinger's inequality. 

\begin{proposition}[{Ozawa~\cite[Theorem~1]{Oz:JFA95}}] \label{prop:Orlicz.critical-Sobolev-embed} 
 The Sobolev space $W^{n/2}_2(\R^n)$ embeds continuously into the Orlicz space $\exp(L_2)(\R^n)$. 
\end{proposition}

Supposed now that $(M^n,g)$ is a closed Riemannian manifold and $(E^r, \acou{\cdot}{\cdot}_E)$ is a Hermitian vector bundle over $M$. We define $L_F(M,E)$ as the set of measurable sections $u:M\rightarrow E$ such that the function $\|u(x)\|_E$  is in $\LlogL(M)$, where $\|\cdot\|_E$ is the norm associated with the Hermitian metric of $E$. We then equip $\LlogL(M,E)$ with the norm, 
\begin{equation*}
 \| u\|_{L_F(M,E)}= \big\| \|u\|_E \big\|_{L_F(M)}, \qquad u\in L_F(M,E)
\end{equation*}
Equivalently, $u\in L_F(M,E)$ if and only if,  for every local chart $\kappa:U\rightarrow V\subset \R^n$ over which there is a trivialization $\tau:E_{|U}\rightarrow U\times \C^r$ and for every test function $\varphi\in C^\infty_c(U)$, the pushforward $(\kappa\circ\tau)_*(\varphi u)$ is in $L_F(\R^n)\otimes \C^r$. 
If $(\varphi_i)_{1\leq i \leq N}$ is a finite smooth partition of unity subordonate to an open cover $(U_i)_{1\leq i \leq N}$ such that each open $U_i$ is the domain of chart $\kappa_i:U_i\rightarrow \R^n$ over which there is a trivialization $\tau_i:E_{|U_i}\rightarrow U_i\times \C^r$, then an equivalent Banach norm is given by
\begin{equation}
 u\longrightarrow \sum_{1\leq i \leq N} \big\| \|(\kappa_i\circ\tau_i)_*(\varphi_i u)\|_{\C^r}\big\|_{L_F}.
 \label{eq:Orlicz.Orlicz-norm-ME}
\end{equation}
Note that $C^\infty(M,E)$ is a dense subspace of $L_F(M,E)$.  

In view of the definitions~(\ref{eq:Orlicz.Sobolev-norm-ME}) and~(\ref{eq:Orlicz.Orlicz-norm-ME}) of the norms on $W_2^{n/2}(M,E)$ and $\exp(L_2)(M,E)$ we have the following consequence of Proposition~\ref{prop:Orlicz.critical-Sobolev-embed}. 

\begin{corollary}\label{cor:critical-Sobolev-embedME} 
 The Sobolev space $W^{n/2}_2(M,E)$ embeds continuously into $\exp(L_2)(M,E)$. 
\end{corollary}

\begin{remark}
 In the scalar case, i.e., $E$ is the trivial line bundle, the above result was obtained by Fontana~\cite{Fo:CMH93} in any dimension $n\geq 2$ and by Branson-Chang-Yang~\cite{BCY:CMP92} in dimension $n=4$.  These papers predate~\cite{Oz:JFA95}. 
\end{remark}

As $e^t-1-t$ and $(1+t)\log(1+t)-t$ are convex conjugates of each other (see~\cite[Example 2.18]{Si:Cambridge11}) we have the following version of H\"older's inequality. 

\begin{proposition}[see~{\cite[Theorem 2.21]{Si:Cambridge11}}]\label{prop:Critical.Holder-Orlicz}
 If $u\in \LlogL(M)$ and $v\in \exp(L_1)(M)$, then $uv\in L_1(M)$ with norm inequality,
\begin{equation*}
 \|uv\|_{L_1} \leq C_{M} \|u\|_{\LlogL} \|v\|_{\exp(L_1)}. 
\end{equation*}
 \end{proposition}

If $v\in  \exp(L_2)(M)$, then $v^2\in \exp(L_1)(M)$ and $\|v^2\|_{\exp{L_1}}\leq C\|v\|_{\exp{L_2}}^2$. A polarization argument shows that if 
$v,w\in  \exp(L_2)(M)$, then $vw\in \exp(L_1)(M)$, and we have  
\begin{equation*}
 \|vw\|_{\exp(L_1)} \leq C \|v\|_{\exp(L_2)} \|w\|_{\exp(L_2)}.
\end{equation*}
 Combining this with Proposition~\ref{prop:Critical.Holder-Orlicz} shows that, if $u\in   \LlogL(M)$ and $v,w\in  \exp(L_2)(M)$, then $uvw\in L_1(M)$ with norm inequality, 
\begin{equation*}
 \|uvw\|_{L_1} \leq C_M \|u\|_{\LlogL} \|v\|_{\exp(L_2)} \|w\|_{\exp(L_2)}.
\end{equation*}

More generally, if $u\in \LlogL(M,\End(E))$ and $v,w\in  \exp(L_2)(M,E)$, then
\begin{equation*}
 \big|\scal{w(x)}{u(x)v(x)}_{E}\big|\leq \|u(x)\|_{\End(E)}\|v(x)\|_{E} \|w(x)\|_{E}\in L_1(M). 
\end{equation*}
Thus, $\scal{w(x)}{u(x)v(x)}_{E}\in L_1(M)$, and we have
\begin{equation*}
 \big\| \scal{w(x)}{u(x)v(x)}_{E}\big\|_{L_1} \leq C_{M}  \|u\|_{\LlogL} \|v\|_{\exp(L_2)} \|w\|_{\exp(L_2)}.
\end{equation*}
Therefore, if we denote by $\exp(L_2)(M,E)^*$ the anti-linear dual of $\exp(L_2)(M,E)$, then $u$ acts as a bounded operator $u: \exp(L_2)(M ,E)\rightarrow \exp(L_2)(M,E)^*$ such that
\begin{equation*}
 \acou{uv}{w}=\int_{M}\scal{w(x)}{u(x)v(x)}_{E_x}\sqrt{g(x)}dx, \qquad v,w\in \exp(L_2)(M,E). 
\end{equation*}
This operator depends continuously on $u$ as $u$ ranges over $ \LlogL(M,E)$. By duality the continuous embedding of $W_2^{n/2}(M,E)$ into $\exp(L_2)(M,E)$ provided by Corollary~\ref{cor:critical-Sobolev-embedME}  gives rise to a continuous embedding of $\exp(L_2)(M,E)^*$ into $W_2^{-n/2}(M,E)$. Therefore, we arrive at the following result.

\begin{proposition}\label{prop:Orlicz.boundedness-LlogL} 
 If $u\in \LlogL(M,\End(E))$, then $u$ gives rise to a bounded operator, 
\begin{equation*}
 u:W_2^{\frac{n}2}(M,E) \longrightarrow W_2^{-\frac{n}2}(M,E). 
\end{equation*}
This operator depends continuously on $u$  with respect to the $\LlogL$-norm. 
\end{proposition}

Combining this with the boundedness properties of \psidos, we obtain the following corollary. 

\begin{corollary}\label{cor:Orlicz.boundednessQuP}
 Let $P$ and $Q$ be operators in $\Psi^{-n/2}(M,E)$. If $u \in \LlogL(M,\End(E))$, then the composition $QuP$ is bounded on $L_2(M,E)$
and depends continuously on $u$ with respect to the $\LlogL$-norm. 
\end{corollary}

\subsection{Cwikel-type estimates} The original estimates of Cwikel~\cite{Cw:AM77} imply that if $f\in L_p(\R^n)$, $n\geq 3$, with $p>1$, then the operator $(1+\Delta)^{-n/2p}f(1+\Delta)^{-n/2p}$ is in the weak trace class $\sL_{p,\infty}$, and we have 
\begin{equation*}
 \big\| (1+\Delta)^{-\frac{n}{4}}f(1+\Delta)^{-\frac{n}{4}}\big\|_{p,\infty} \leq C_n \|f\|_{L_p}. 
\end{equation*}
We shall now explain how to extend these estimates to operators of the form $QuP$ where $P,Q\in \Psi^{-n/2}(M,E)$ and $u \in \LlogL(M,E)$. Recall that such operators are bounded by Corollary~\ref{cor:Orlicz.boundednessQuP}. 

The starting point is the following specific Cwikel-type estimate on the torus $\T^n=\R^n/\Z^n$. 

\begin{proposition}[Sukochev-Zanin~\cite{SZ:arXiv21}; see also~\cite{So:PLMS95}] \label{prop:Critical.Cwikel} 
 If $f\in \LlogL(\T^n)$, then $(1+\Delta)^{-n/4}f(1+\Delta)^{-n/4}$ is in the weak trace class $\sL_{1,\infty}$, and we have 
\begin{equation*}
 \big\| (1+\Delta)^{-\frac{n}{4}}f(1+\Delta)^{-\frac{n}{4}}\big\|_{\sL_{1,\infty}} \leq C_n  \|f\|_{\LlogL}. 
\end{equation*}
\end{proposition}

There is no difficulty to extend the above result to pseudodifferential operators. 

\begin{corollary}\label{cor:Orlicz.Cwikel-Tn-PsiDOs} 
 Let $P$ and $Q$ be operators in $\Psi^{-n/2}(\T^n)$. If $f \in \LlogL(\T^n)$, then the operator $QfP$ is in the weak trace class $\sL_{1,\infty}$, and we have 
\begin{equation}
 \big\| QfP\big\|_{\sL_{1,\infty}} \leq C_n \|P\|_{\sL(L_2, W_2^{{n}/{2}})}\|Q\|_{\sL(L_2, W_2^{{n}/{2}})}  \|f\|_{\LlogL}. 
 \label{eq:Orlicz.Cwikel-Tn-PsiDOs} 
\end{equation}
\end{corollary}
\begin{proof}
 The operator $A=(1+\Delta)^{n/4}P$ and $B=Q(1+\Delta)^{n/4}$ are zeroth order \psidos, and hence are bounded on $L_2(M,E)$. As $QfP=B(1+\Delta)^{-n/4}f(1+\Delta)^{-n/4}A$, by using Proposition~\ref{prop:Critical.Cwikel}  we deduce that $QfP\in \sL_{1,\infty}$, and we have 
\begin{align*}
  \big\| QfP\big\|_{\sL_{1,\infty}} & \leq \|B\| \big\| (1+\Delta)^{-\frac{n}{4}}f(1+\Delta)^{-\frac{n}{4}}\big\|_{\sL_{1,\infty}} \|A\|\\ 
  & \leq C_n\|A\|_{\sL(L_2)} \|B\|_{\sL(L_2)}  \|f\|_{\LlogL}.  
\end{align*}
Note that
\begin{gather*}
 \|A\|_{\sL(L_2)} =\big\|(1+\Delta)^{n/4}P\big\|_{\sL(L_2)} =\|P\|_{\sL(L_2, W_2^{{n}/{2}})},\\
 \|B\|_{\sL(L_2)} =\|Q(1+\Delta)^{n/4}\|_{\sL(L_2)} =\|Q\|_{\sL(W_2^{-{n}/{2}} ,L_2)}= \|Q\|_{\sL(L_2, W_2^{{n}/{2}})}. 
\end{gather*}
This gives the estimate~(\ref{eq:Orlicz.Cwikel-Tn-PsiDOs}). The proof is complete. 
\end{proof}

As \psidos\ are (pseudo-)local objects it's a routine argument to extend results on a given closed manifold to all such manifolds.

From now on, we let $M^n$ be a closed Riemannian manifold and $E^r$ a Hermitian vector bundle over $M$. In the following, given any compact $K\subset M$, we denote by $\Psi_K^m(M,E)$, $m\in \R$, the class of operators $P\in \Psi^m(M,E)$ whose Schwartz kernels are supported in $K\times K$. We also denote by 
$L_{2,K}(M,E)$ (resp., $\LlogL_K(M,\End(E))$) the sections $u$ in $L_2(M,E)$ (resp., $\LlogL(M,\End(E))$) such that $\supp u\subset K$. 

Corollary~\ref{cor:Orlicz.Cwikel-Tn-PsiDOs} admits the following local version.

\begin{lemma}\label{lem:Orlicz.Cwikel-chart} 
 Suppose that $K$ is a compact subset of $M$ which is contained in the domain of a local chart of $M$ over which $E$ is trivial. Let $P$ and $Q$ be operators in $\Psi_K^{-n/2}(M,E)$. If $u\in \LlogL_K(M,\End(E))$, then $QuP$ is in the weak trace class $\sL_{1,\infty}(L^2(M,E))$, and we have
\begin{equation}
 \big\|QuP\big\|_{1,\infty} \leq C_{KPQ} \|u\|_{\LlogL}.
 \label{eq:Orlicz.Cwikel-chart}  
\end{equation}
\end{lemma}
\begin{proof}
As $\R^n$ can be smoothly embedded as an open set of $\T^n$, the assumptions ensure us there exist an open set $U\subset M$ containing $K$ over which $E$ is trivializable and a smooth diffeomorphism $\phi:U\rightarrow V$ where $V$ is an open set of $\T^n$. Set $K'=\phi(K)$, and let
$\tau:E_{|U}\rightarrow U\times \C^r$ be a smooth trivialization of $E$ over $U$. We then have the following continuous linear embeddings of Hilbert spaces,  
\begin{equation*}
 L_2(M,E) \xleftarrow[(\iota_U)_*]{~} L_2(U,E_{|U})\xrightarrow[~\tau_*~]{\sim}L_{2,K}(U)\otimes \C^r \xrightarrow[~\phi_*~]{\sim} L_{2,K'}(V)\otimes \C^r \xrightarrow[(\iota_V)_*]{~}L_2(\T^n)\otimes \C^r. 
\end{equation*}
Here $\tau_*:L_{2,K}(U,E)\otimes \C^r\rightarrow L_{2,K}(U)\otimes \C^r$ is the pushforward-isomorphism given by the trivialization $\tau$. The linear embeddings $(\iota_U)_*:L_{2,K}(U, E)\rightarrow L_2(M,E)$ and $(\iota_V)_*:L_{2,K'}(V)\otimes \C^r\rightarrow L_2(\T^n)\otimes \C^r$ arise from the inclusions $U\subset M$ and $V\subset \T^n$, respectively. They are just given by the extension to zero outside $K$ and $K'$, and hence are isometric embeddings. Therefore, by Proposition~\ref{prop:App.eigenvalues} in Appendix there is $c>0$ such that,  for any compact operator $A$ on $L_{2,K}(U,E)$,  we have 
\begin{equation}
 \mu_j\big((\iota_U)_*A)=\mu_j(A)\leq c\mu_j\big((\iota_V)_*\phi_*\tau_*A\big) \qquad \forall j\geq 0.
 \label{eq:Orlicz.estimates-sing-values}  
\end{equation}

Set $\tilde{P}=(\iota_V)_*\phi_*\tau_*(P_{|U})$ and $\tilde{Q}=(\iota_V)_*\phi_*\tau_*(Q_{|U})$. They both are operators in $\Psi^{-n/2}(\T^n)$. If $\psi\in C^\infty_c(V)$ is such that $\psi=1$ near $K'$, then $\tilde{P}=\psi(\phi_*\tau_*(P_{|U}))\psi$ and $\tilde{Q}=\psi(\phi_*\tau_*(Q_{|U}))\psi$. 
If $u\in \LlogL_K(M, \End(E))$, then $\tilde{u}:=(\iota_V)_*\phi_*\tau_*(u_{|U})\in \LlogL_{K'}(\T^n)\otimes M_r(\C)$, and so Corollary~\ref{cor:Orlicz.Cwikel-Tn-PsiDOs}  ensures that $\tilde{Q}\tilde{u}\tilde{P}\in \sL_{1,\infty}(L^2(\T^n)\otimes \C^r)$, and we have
\begin{equation*}
 \big\|\tilde{Q}\tilde{u} \tilde{P}\big\|_{1,\infty} \leq C_{\tilde{P}\tilde{Q}}\|\tilde{u}\|_{\LlogL} \leq C_{KPQ} \|u\|_{\LlogL}. 
\end{equation*}
Moreover, as $PuQ=(\iota_U)_*(Q_{|U}u_{|U}P_{|U})$ and $\tilde{Q}\tilde{u}\tilde{P}=(\iota_V)_*\phi_*\tau_*(Q_{|U}u_{|U}P_{|U})$, by~(\ref{eq:Orlicz.estimates-sing-values}) we have
\begin{equation*}
 \mu_j(PuQ)\leq c\mu_j(\tilde{Q}\tilde{u}\tilde{P})\qquad \forall j\geq 0. 
\end{equation*}
It then follows that $QuP$ is in the weak Schatten class $\sL_{1,\infty}(L(M,E))$, and the estimate~(\ref{eq:Orlicz.Cwikel-chart}) holds. The proof is complete. 
\end{proof}

\begin{remark}
 If $n$ is even, then Lemma~\ref{lem:Orlicz.Cwikel-chart} can also be deduced from the Cwikel-type estimates of Solomyak~\cite{So:PLMS95}.
\end{remark}

\begin{remark}
 Lemma~\ref{lem:Orlicz.Cwikel-chart} continues to holds if $M$ is not compact. 
\end{remark}

\begin{lemma}\label{lem:Orlicz.Cwikel-smoothing} 
Assume $M$ is a closed manifold. Let $P$ and $Q$ be operators in $\Psi^{-n/2}(M,E)$ such that one of them is smoothing. If $u\in \LlogL(M,\End(E))$, then $QuP$ is in every weak Schatten class $\sL_{p,\infty}$, $p>0$, and we have
\begin{equation}
 \| QuP\|_{p,\infty} \leq C_{PQp} \|u\|_{\LlogL}. 
 \label{eq:Orlicz.Cwikel-PQ-smoothing} 
\end{equation}
\end{lemma}
\begin{proof}
 Let $ p>0$. Assume that $Q$ is smoothing. Let $\Delta_E=\nabla^*\nabla$ be the Laplacian of some Hermitian connection on $E$. If $u\in \LlogL(M,\End(E))$, then we have
\begin{equation*}
 QuP= \big( Q(1+\Delta_E)^{\frac{n}4}\big) \cdot (1+\Delta_E)^{-\frac{n}4}uP. 
\end{equation*}
Here $Q(1+\Delta_E)^{\frac{n}4}$ is a smoothing operator, and so it is contained in  every weak Schatten class $\sL_{p,\infty}$. Moreover, as $(1+\Delta_E)^{-\frac{n}4}$ is a \psido\ of order~$-n/2$, Corollary~\ref{cor:Orlicz.boundednessQuP} ensures that $(1+\Delta_E)^{-\frac{n}4}uP$ is bounded on $L_2(M,E)$ and depends continuously on $u$. It follows that $QuP$ is in $\sL_{p,\infty}$ and depends continuously on $u$, i.e., we have the estimate~(\ref{eq:Orlicz.Cwikel-PQ-smoothing}). Similar arguments give the result when $P$ is smoothing. The proof is complete. 
\end{proof}

We are now in a position to get the Cwikel-type estimates we are seeking for.  

\begin{proposition}\label{prop:Orlicz.Cwikel-ME-PsiDOs} 
 Suppose that $M^n$ is a closed manifold and $E$ is a Hermitian vector bundle over $M$. Let $P$ and $Q$ be operators in $\Psi^{-n/2}(M,E)$. If $u\in \LlogL(M,\End(E))$, then the operator $QuP$ is in the weak trace class  $\sL_{1,\infty}(L^2(M,E))$, and we have
\begin{equation}
  \|QuP\|_{1,\infty} \leq C_{PQ} \|u\|_{\LlogL}.
  \label{eq:Orlicz.Cwikel-QuP-ME}
\end{equation}
\end{proposition}
\begin{proof} This is a routine partition of unity argument (see, e.g.,~\cite{BS:SMJ79}). 
Let $(\varphi_i)_{1\leq i \leq N}$ be a smooth partition of unity subordinate to an open cover $(U_i)_{1\leq i \leq N}$, where each open $U_i$ is the domain of a chart over which $E$ is trivializable. For $i=1, \ldots, N$, let $\psi_i\in C^\infty_c(U_i)$ be such that $\psi_i=1$ near $\supp \varphi_i$. Given any $u\in \LlogL(M,\End(E))$ we have
\begin{align*}
 QuP = \sum_{1\leq i \leq N} Q\varphi_i u P = \sum_{1\leq i \leq N} \bigg\{ Q_i (\varphi_i u ) P_i + S_i uP + Q_i uR_i\bigg\},
\end{align*}
where we have set
\begin{equation*}
 P_i=\psi_i P \psi_i, \quad Q_i=\psi_i Q \psi_i,\quad R_i=\varphi_i P(1-\psi_i), \quad S_i=(1-\psi_i)Q\varphi_i. 
\end{equation*}
Here $P_i$ and $Q_i$ are operators in $\Psi_{K_i}(M,E)$ with $K_i=\supp \psi_i\subset U_i$. Thus, by Lemma~\ref{lem:Orlicz.Cwikel-chart} each operator $Q(\varphi_i u) P$ is in $\sL_{1,\infty}$, and we have 
\begin{equation*}
 \|Q(\varphi_i u) P\|_{1,\infty} \leq C_{P_iQ_iK_i} \|\varphi_i u\|_{\LlogL}  \leq C_{P_iQ_iK_i} \| u\|_{\LlogL} . 
\end{equation*}
Note also that $R_i$ and $S_i$ are both smoothing operators. Therefore, by Lemma~\ref{lem:Orlicz.Cwikel-smoothing}  the operators $S_i uP$ and $Q_i uR_i$ are both in $\sL_{1,\infty}$, and we have 
\begin{equation*}
 \|S_i uP\|_{1,\infty} \leq C_{S_iP} \| u\|_{\LlogL}, \qquad  \|Q_i uR_i\|_{1,\infty} \leq C_{Q_i R_i} \| u\|_{\LlogL}. 
\end{equation*}
Combining all this we deduce that $QuP$ is in the weak trace class  $\sL_{1,\infty}(L^2(M,E))$, and we have
\begin{align*}
   \|QuP\|_{1,\infty} & \leq  \sum_{1\leq i \leq N} \bigg\{ \|Q_i (\varphi_i u ) P_i\|_{1,\infty} + \|S_i uP\|_{1,\infty} + \|Q_i uR_i\|_{1,\infty}\bigg\}\\
   & \leq C   \| u\|_{\LlogL}. 
\end{align*}
The proof is complete. 
\end{proof}

\begin{remark}
 In the scalar case, building on the results of~\cite{RS:EMS21}, Rozenblum~\cite[Corollary 3.4]{Ro:arXiv21} established eigenvalues inequalities for operators of the form $P^*uP$, where $P\in \Psi^{-n/2}(M)$ and $u$ is a (real-valued) potential of the form $u=f\mu$ such that  $\mu$ is an Alfhors-regular measure supported on a  regular submanifold $\Sigma \subset M$ and $f$ is in $\LlogL(\Sigma, \mu)$ (see also~\cite{RT:arXiv21} for the non-critical case). These eigenvalues inequalities can be used to get Cwikel-type inequalities for this more general kind of potentials. 
 \end{remark}

Proposition~\ref{prop:Orlicz.Cwikel-ME-PsiDOs} implies the following specific Cwikel-type estimates.  

\begin{corollary}\label{cor:Orlicz.Cwikel-Delta}
 Under the assumptions of Proposition~\ref{prop:Orlicz.Cwikel-ME-PsiDOs}, let $\Delta_E=\nabla^*\nabla$ be the Laplacian of some Hermitian connection $\nabla$ on $E$. If 
 $u\in \LlogL(M,\End(E))$, then the operator $\Delta_E^{-n/4} u\Delta_E^{-n/4}$ is in the weak trace class  $\sL_{1,\infty}(L^2(M,E))$, and we have
\begin{equation}
  \big\|\Delta_E^{-\frac{n}{4}} u\Delta_E^{-\frac{n}{4}}\big\|_{1,\infty} \leq C_{\nabla} \|u\|_{\LlogL}. 
  \label{eq:Orlicz.Cwikel-Delta}
\end{equation}
\end{corollary}

\begin{remark}
 In the scalar case, Sukochev-Zanin~\cite[Theorem~1.1]{SZ:arXiv21} obtained Cwikel-type estimates for operators of the form $(1+\Delta_g)^{-\frac{n}{4}}f(1+\Delta_g)^{-\frac{n}{4}}$ with $f\in \LlogL(M)$. We recover those estimates by specializing Proposition~\ref{prop:Orlicz.Cwikel-ME-PsiDOs} to 
 $P=Q=(1+\Delta_g)^{-\frac{n}{4}}$. 
\end{remark}

\section{Weyl's Laws and Semiclassical Analysis} 
\label{sec:Weyl-Orlicz}  
In this section, we explain how the Cwikel-type estimates of the previous section lead to Weyl's laws for operators of the form $QuP$, where  $u$ is a matrix-valued $\LlogL$-Orlicz potential and $P$ and $Q$ are \psidos\ of order $-n/2$ on an $n$-dimensional closed manifold (compare~\cite{SZ:arXiv21}). 
Thanks to the  Birman-Schwinger principle this will further allow us to get semiclassical Weyl's laws and CLR-type inequalities for the number of negative eigenvalues (i.e., the number of bound states) of critical Schr\"odinger operators $\Delta_E^{n/2}+V$, where $V$ is a matrix-valued  $\LlogL$-Orlicz potential.  

Throughout this section, we let $(M^n,g)$ be a closed Riemannian manifold and $E$ a Hermitian vector bundle over $M$. 

\subsection{Weyl's laws} The  following is the main result of this section. 

\begin{theorem}\label{thm:Orlicz.WeylQuP}
Let $P$ and $Q$ be operators in $\Psi^{-n/2}(M,E)$. If $u\in \LlogL(M,\End(E))$, then the following holds. 
\begin{enumerate}
 \item The operators $QuP$ and $|QuP|$ are both Weyl operators in $\sL_{1,\infty}$. 
 
 \item We have 
 \begin{equation}
\lim_{j\rightarrow \infty} j\mu_j\left(QuP\right) = \frac1{n} (2\pi)^{-n} \!  \int_{S^*M} \tr_E \left\{\big|\sigma(Q)(x,\xi)u(x) \sigma(P)(x,\xi)\big|\right\}dxd\xi.
 \label{eq:Orlicz.Weyl-|QuP|} 
\end{equation}

\item Suppose that $Q=P^*$ and $u(x)^*=u(x)$. Then $P^*uP$ is selfadjoint, and we have
\begin{equation}
  \lim_{j\rightarrow \infty} j\lambda^\pm_j\left(P^*uP\right)=  \frac1{n} (2\pi)^{-n} \!  \int_{S^*M} \tr_E\left\{ \big[ \sigma(P)(x,\xi)^*u(x) \sigma(P)(x,\xi)\big]_\pm \right\} dxd\xi.
  \label{eq:Orlicz.Weyl-laws} 
\end{equation}
\end{enumerate}
 \end{theorem}
\begin{proof}
It follows from Proposition~\ref{prop:Bir-Sol.closedness} and Proposition~\ref{prop:Bir-Sol.closedness-sing} that $\sW_{1,\infty}\cap \sW_{|1,\infty|}$ is a closed subset of $\sL_{1,\infty}$ on which
$A\rightarrow \Lambda_\pm(A)$ and $A\rightarrow\Lambda(|A|)$ are continuous functionals. It also follows from Proposition~\ref{prop:Bir-Sol-asymp}
that $\sW_{1,\infty}\cap\sW_{|1,\infty|}$ contains the subspace $ \sV:=\left\{ QvP; v \in C^\infty(M,\End(E))\right\}$, and hence contains its closure $\overline{\sV}$. Proposition~\ref{prop:Orlicz.Cwikel-ME-PsiDOs} implies that $u\rightarrow QuP$ is a continuous linear map from $\LlogL(M,\End(E))$ to $\sL_{1,\infty}$. As $C^\infty(M,\End(E))$ is dense in $\LlogL(M,E)$, it follows that if $u\in   \LlogL(M,\End(E))$, then $QuP$ is contained in the closure $\overline{\sV}$, and hence is contained in $\sW_{1,\infty}\cap \sW_{|1,\infty|}$. That is, $QuP$ and $|QuP|$ are both Weyl operators in $\sL_{1,\infty}$.

Furthermore,  it follows from Proposition~\ref{prop:Bir-Sol-asymp} that, for all $u\in C^\infty(M,\End(E))$, we have
\begin{equation*}
 \Lambda\big(|QuP|)= \frac1{n} (2\pi)^{-n} \!  \int_{S^*M} \tr_E \left\{\big|\sigma(Q)(x,\xi)u(x) \sigma(P)(x,\xi)\big|\right\}dxd\xi.
\end{equation*}
Each side depends continuously on $u$ with respect to the $\LlogL$-norm. As $C^\infty(M,\End(E))$ is dense in  $\LlogL(M,\End(E))$ the equality continues to hold for any $u\in \LlogL(M,\End(E))$. This gives~(\ref{eq:Orlicz.Weyl-|QuP|}). 
Likewise, if $Q=P^*$, then, again by Proposition~\ref{prop:Bir-Sol-asymp}, the equality (\ref{eq:Orlicz.Weyl-laws}) holds for all $u\in  C^\infty(M,\End(E))$ and each side depends continuously on $u$ with respect to the $\LlogL$-norm. Therefore, the equality continues to holds for any  $u\in \LlogL(M,\End(E))$. The proof is complete. 
\end{proof}
 
\begin{remark}\label{rmk:Weyl-QuP} 
The Weyl's law~(\ref{eq:Orlicz.Weyl-|QuP|}) is new. In the scalar case, Rozenblum~\cite{Ro:arXiv21} established the Weyl's law~(\ref{eq:Orlicz.Weyl-laws}) 
for a larger class of potentials of the form $u=f\mu$, where  $\mu$ is an Alfhors-regular measure supported on a  regular submanifold $\Sigma \subset M$ and $f$ is a real-valued function in $\LlogL(\Sigma, \mu)$ (see also~\cite{RT:arXiv21} for the non-critical case). 
\end{remark}
  
 Let $\Delta_E=\nabla^*\nabla$ be the Laplacian of some Hermitian connection $\nabla$ on $E$. In particular, $\Delta_E$ is (formally) selfadjoint and has principal symbol $\sigma(\Delta_E)=|\xi|^2$ (where we denote by $|\cdot|$ the Riemannian metric on $T^*M$). Specializing Theorem~\ref{thm:Orlicz.WeylQuP} to $P=\Delta_E^{-n/4}$ immediately leads to the following statement.

\begin{corollary}\label{thm:Orlicz.int-formulaDelta}
 Let  $u\in \LlogL(M,\End(E))$. The following holds. 
 \begin{enumerate}
\item The operators $\Delta_E^{-n/4}u\Delta_E^{-n/4}$ and $|\Delta_E^{-n/4}u\Delta_E^{-n/4}|$ are Weyl operators in $\sL_{1,\infty}$.

\item We have 
\begin{equation}
   \lim_{j\rightarrow \infty} j\mu_j\left(\Delta_E^{-\frac{n}{4}}u\Delta_E^{-\frac{n}{4}}\right)= \frac1{n} (2\pi)^{-n} \! \int_M  \tr_E\big[|u(x)|\big] \sqrt{g(x)}dx.
   \label{eq:Orlicz.Weyl-Delta-mu}
 \end{equation}

\item If $u(x)^*=u(x)$, then 
\begin{equation}
  \lim_{j\rightarrow \infty} j\lambda^\pm_j\left(\Delta_E^{-\frac{n}{4}}u\Delta_E^{-\frac{n}{4}}\right)= \frac1{n} (2\pi)^{-n} \! \int_M  \tr_E\big[u(x)_\pm\big] \sqrt{g(x)}dx.
  \label{eq:Orlicz.Weyl-Deltapm}  
\end{equation}
\end{enumerate}
\end{corollary}

\begin{remark}
 In the scalar case, Sukochev-Zanin~\cite[Theorem~1.2]{SZ:arXiv21} obtained a Weyl's law similar to~(\ref{eq:Orlicz.Weyl-Deltapm}) for operators of the form $(1+\Delta_g)^{-n/4}f (1+\Delta_g)^{-n/4}$ with $f\in  \LlogL(M)$. The approach of~\cite{SZ:arXiv21} involves deep results on commutators in operator ideals from~\cite{DFWW:AIM04, HSZ:Preprint}. This is unnecessary, since, as the proof of Theorem~\ref{thm:Orlicz.WeylQuP} displays, the Weyl's laws~(\ref{eq:Orlicz.Weyl-|QuP|})--(\ref{eq:Orlicz.Weyl-Deltapm}) are direct consequences of the Birman-Solomyak's Weyl's laws~(\ref{eq:Weyl.Bir-Sol-mu})--(\ref{eq:Weyl.Bir-Sol-selfadjoint}) and the Cwikel-type estimates~(\ref{eq:Orlicz.Cwikel-QuP-ME}). 
\end{remark}

\subsection{Semiclassical Weyl's law}\label{sec:SC}  
As above we let $\Delta_E=\nabla^*\nabla$ be the Laplacian of some Hermitian connection on $E$. The operator $\Delta_E^{n/2}$ with domain $W_2^{n}(M,E)$ is selfadjoint with non-negative and discrete spectrum. Let $V(x)=V(x)^*\in \LlogL(M,E)$. This defines a continuous operator $V:W^{n/2}(M,E)\rightarrow W^{-n/2}(M,E)$ (\emph{cf}.\ Proposition~\ref{prop:Orlicz.boundedness-LlogL}). Corollary~\ref{cor:Orlicz.Cwikel-Delta} implies that $\Delta_E^{-n/4}V\Delta_E^{-n/4}$ is a compact operator (and even is weak trace-class). This means that $V$ is relatively $\Delta_E^{n/2}$-form compact, and so we may define the (fractional) Schr\"odinger operator $H_V:=\Delta_E^{n/2}+V$ as a form sum (see, e.g., \cite{Si:AMS15}). Namely, $H_V$ is the selfadjoint operator defined by the closed quadratic form,
\begin{equation*}
 Q_{H_V}(\xi,\xi)=\bigacou{\Delta_E^{\frac{n}2}\xi}{\xi} +\acou{V\xi}{\xi}, \qquad \xi \in W^{\frac{n}{2}}(M,E), 
\end{equation*}
where $\acou{\cdot}{\cdot}:W^{n/2}(M,E)\rightarrow W^{-n/2}(M,E)$ is the duality pairing. 

The operator $H_V$ is bounded from below and has discrete spectrum (see, e.g., \cite{Si:AMS15}). Therefore, we can arrange its spectrum as a non-decreasing sequence, 
\begin{equation*}
 -\infty < \lambda_0\big(H_V\big) \leq  \lambda_1\big(H_V\big) \leq  \lambda_2\big(H_V\big)\leq \cdots, 
\end{equation*}
where each eigenvalue is repeated according to multiplicity. We are interested in the number of negative eigenvalues, i.e., 
\begin{equation*}
 N^{-}(H_V):= \# \left\{ j; \lambda_j\big(H_V\big)<0\right\}. 
\end{equation*}
 In physicists' language $N^{-}(H_V)$ is the number of bound states of $H_V$. In addition, if we denote by $\sF^{-}(H_V)$ the collection of subspaces $F\subset W^{n/2}(M,E)$ such that $Q_{H_V}(\xi,\xi)<0 $ on $F\setminus 0$, then by Glazman's lemma~\cite{BS:Book} we have
\begin{equation}
 N^{-}(H_V)= \max\left\{ \dim F; \ F\in \sF^-(H_V)\right\}. 
 \label{eq:SC.variational-principle}
\end{equation}
This implies the monotonicity principle, 
\begin{equation}
 V_1\leq V_2 \Longrightarrow N^{-}(H_{V_2}) \leq  N^{-}(H_{V_1}) .
 \label{eq:Weyl.monotonicity}  
\end{equation}

Recall if $A$ is a selfadjoint compact operator, then $N^{\pm}(A;\lambda)=\#\{j;\ \lambda_j^\pm(A)>\lambda\}$, $\lambda>0$ (\emph{cf.}\ Remark~\ref{rmk:Bir-Sol.counting}). By the abstract Birman-Schwinger principle~\cite{BS:AMST89} (see also~\cite{MP:Part1}), we have
\begin{equation}
N^{-}\big( \Delta_E^{-\frac{n}{4}}V\Delta_E^{-\frac{n}{4}};1\big) \leq N^{-}\big( \Delta_E^{\frac{n}{2}}+V\big) \leq N^{-}\big( \Delta_E^{-\frac{n}{4}}V\Delta_E^{-\frac{n}{4}};1\big)+  \dim \ker \Delta_E.
\label{eq:Weyl.ABSP}
\end{equation}
If in addition $V(x)\leq 0$, then by the borderline Birman-Schwinger principle of~\cite{MP:Part1} we have 
\begin{equation}
 0\leq N^{-}\big( \Delta_E^{\frac{n}{2}}+V\big) - N^+(\Pi_0V\Pi_0) \leq  \big\|\Delta_E^{-\frac{n}{4}} V_{-}\Delta_E^{-\frac{n}{4}}\big\|_{1,\infty}, 
 \label{eq:Weyl.BBSP} 
\end{equation}
where $\Pi_0$ is the orthogonal projection onto $\ker \Delta_E$. 

It follows from~(\ref{eq:Weyl.ABSP}) (see, e.g.,~\cite{BS:TMMS72, BS:AMST80, Po:NCIntegration}) that, under the semiclassical limit $h\rightarrow 0^+$, we get
\begin{equation}
 N^{-}\big(h^n \Delta_E^{\frac{n}{2}}+V\big)= N^{-}\big( \Delta_E^{-\frac{n}{4}}V\Delta_E^{-\frac{n}{4}};h^{n}\big) +\op{O}(1).
  \label{eq:SC.SC-BSP}
\end{equation}
Thanks to~(\ref{eq:Bir-Sol.counting-Lambda}) and the Weyl's law~(\ref{eq:Orlicz.Weyl-Deltapm}) we have
\begin{equation*}
 \lim_{\lambda \rightarrow 0^+} \lambda N^{-}\big( \Delta_E^{-\frac{n}{4}}V\Delta_E^{-\frac{n}{4}};\lambda\big) =  
 \frac1{n} (2\pi)^{-n} \! \int_M  \tr_E\big[V(x)_{-}\big] \sqrt{g(x)}dx. 
\end{equation*}
Combining this with~(\ref{eq:SC.SC-BSP}) we then arrive at the following semiclassical Weyl's law. 

\begin{corollary}\label{cor:SC.SC-Weyl}
 If $V(x)=V(x)^*\in \LlogL(M, E)$, then 
 \begin{equation}
 \lim_{h \rightarrow 0^+} h^nN^{-}\big(h^n \Delta_E^{\frac{n}{2}}+V\big)= \frac1{n} (2\pi)^{-n} \! \int_M  \tr_E\big[V(x)_{-}\big] \sqrt{g(x)}dx.
 \label{eq:SC.SC-Weyl} 
\end{equation}
\end{corollary}

\begin{remark}
The semiclassical Weyl's law~(\ref{eq:SC.SC-Weyl}) does not seem to have appeared elsewhere, and this sense this is a new result. However, as mentioned in Remark~\ref{rmk:Weyl-QuP}, in the scalar case Rozenblum~\cite{Ro:arXiv21} established a version of the Weyl's law~(\ref{eq:Orlicz.Weyl-laws}) for potentials of the form $V=f\mu$, where $\mu$ is an Alfhors-regular measure supported on a  regular submanifold $\Sigma \subset M$ and $f$ is a real-valued function in 
$\LlogL(\Sigma,\mu)$. Using this result and arguing as above it is immediate to obtain we version of the semiclassical Weyl's law~(\ref{eq:SC.SC-Weyl}) for that class of potentials. 
\end{remark}

\subsection{CLR-Type inequality} 
The celebrated CLR inequality was established by Rozenblum~\cite{Ro:SMD72, Ro:SM76}, Cwikel~\cite{Cw:AM77} and Lieb~\cite{Li:BAMS76, Li:PSPM80} for Schr\"odinger operators $\Delta+V$ on $\R^n$, $n\geq 3$, for potentials $V\in L_{n/2}(\R^n)$ (see~\cite{Co:RMJM85, Fe:BAMS83, Fr:JST14, HKRV:arXiv18, LY:CMP83} for further alternative proofs). It gives a universal bound on the number of negative eigenvalues in terms of the $L_{n/2}$-norm of $V_{-}$. This stem from a conjecture of Simon~\cite{Si:TAMS76} (who was unaware of Rozenblum's results). There are also versions of the CLR inequality for fractional Schr\"odinger operators $\Delta^{n/2p}+V$ with $V\in L_p$ for $p>1$ 
(see, e.g., \cite{Da:CMP83, LS:JAM97, Ro:SMD72, Ro:SM76, RS:SPMJ98, RT:arXiv21}).  In the critical case $p=1$ Solomyak~\cite{So:PLMS95} obtained a CLR-type inequality for Schr\"odinger operators $\Delta^{n/2p}+V$ associated with $\LlogL$-Orlicz potentials $V$ on bounded domains of $\R^n$ in even dimension only. We stress that the inequality does not hold for any larger class of Orlicz potentials, including the class of $L_1$-potentials (see~\cite{So:PLMS95}).

Bearing this in mind, as a further application of the Cwikel estimates~(\ref{eq:Orlicz.Cwikel-Delta}), we have the following CLR-type inequality for matrix-valued $\LlogL$-Orlicz potentials. Note that this result holds in any dimension.  
 
\begin{corollary}[CLR-Type Inequality]\label{cor:SC.CLR} 
 Under the above assumptions, there is a constant $C_\nabla>0$ such that, for every potential $V(x)=V(x)^*\in \LlogL(M,\End(E))$, we have
\begin{equation}
 N^{-}\big( \Delta_E^{\frac{n}{2}}+V\big) - N^+(\Pi_0V_{-}\Pi_0) \leq C_{\nabla}\|V_{-}\|_{\LlogL},  
 \label{eq:Weyl.CLR} 
\end{equation}
 where $\Pi_0$ is the orthogonal projection onto $\ker \Delta_E$. 
\end{corollary}
\begin{proof}
 As $V\geq -V_{-}$ from~(\ref{eq:Weyl.monotonicity}) we get $N^{-}( \Delta_E^{\frac{n}{2}}+V) \leq N^{-}( \Delta_E^{\frac{n}{2}}-V_{-})$, and so by using~(\ref{eq:Weyl.BBSP}) we obtain
\begin{equation*}
 N^{-}\big( \Delta_E^{\frac{n}{2}}+V\big)-N^+(\Pi_0V_{-}\Pi_0) \leq  \big\|\Delta_E^{-\frac{n}{4}} V_{-}\Delta_E^{-\frac{n}{4}}\big\|_{1,\infty} . 
\end{equation*}
Combining this with the Cwikel-type estimate~(\ref{eq:Orlicz.Cwikel-Delta}) then gives the result. 
\end{proof}

\begin{remark}
 As $N^+(\Pi_0V_{-}\Pi_0)\leq \ran \Pi_0=\dim \ker \Delta_E$, the inequality~(\ref{eq:Weyl.CLR}) implies that
 \begin{equation*}
  N^{-}\big( \Delta_E^{\frac{n}{2}}+V\big) - \dim \ker \Delta_E\leq C_{\nabla}\|V_{-}\|_{\LlogL}. 
\end{equation*}
This inequality can also be obtained by using the usual abstract Birman-Schwinger principle of~\cite{BS:AMST89}. Note however that the above inequality is not as sharp as~(\ref{eq:Weyl.CLR}), since there are many examples for which $N^+(\Pi_0V_{-}\Pi_0)\leq \dim \ker \Delta_E-1$ (see, e.g., \cite{MP:Part1}). 
\end{remark}

\begin{remark}
 The semiclassical Weyl's law~(\ref{eq:SC.SC-Weyl}) and the CLR-type estimate~(\ref{eq:Weyl.CLR}) hold \emph{verbatim} if we replace $\Delta_E$ by any Laplace-type operator. In particular, these results hold for squares of Dirac-type operators. 
\end{remark}

\section{Noncommutative Geometry and Weyl's Laws} \label{sec:NCG} 
In this section, we review the notion of integral in the framework of Connes' noncommutative geometry and its relationship to Birman-Solomyak's Weyl's laws for compact operators.  

\subsection{Quantized calculus} 
The main goal of the quantized calculus of Connes~\cite{Co:NCG} is to translate into a the Hilbert space formalism of quantum mechanics the main tools of the classical infinitesimal calculus. 

\renewcommand{\arraystretch}{1.2}

\begin{center}
    \begin{tabular}{c|c}  
        Classical & Quantum \\ \hline       
       Complex variable & Operator on $\sH $  \\
      Real variable &  Selfadjoint operator on $\sH $  \\  
 Infinitesimal variable & Compact operator on $\sH $ \\
       Infinitesimal of order $\alpha>0$  & Compact operator $T$ such that\\ 
                      &  $\mu_{j}(T)=\op{O}(j^{-\alpha})$\\
    \end{tabular}
\end{center}

The first two lines arise from quantum mechanics. Intuitively speaking, an infinitesimal is meant to be smaller than any real number. For a bounded operator the condition $\|T\|<\epsilon$ for all $\epsilon>0$ gives $T=0$. This condition can be relaxed into the following: For every $\epsilon>0$ there is a finite-dimensional subspace $E$ of 
$\sH$ such that $\|T_{|E^\perp}\|<\epsilon$. This is equivalent to $T$ being a compact operator.

The order of compactness of a compact operator is given by the order of decay of its singular values. Namely, an \emph{infinitesimal operator} of order $\alpha>0$ is any compact operator such that $\mu_j(T)=\op{O}(j^{-\alpha})$. Thus, if we set $p=\alpha^{-1}$, then $T$ is {infinitesimal operator} of order $\alpha>0$ iff $T\in \sL_{p,\infty}$. 

The next line of the dictionary is the NC analogue of the integral. As an Ansatz the NC integral should be a linear functional satisfying the following conditions: 
\begin{enumerate}
 \item[(1)] It is defined on a suitable class of infinitesimal operators of order~1. 
 
 \item[(2)] It vanishes on infinitesimal operators of order~$> 1$. 
 
 \item[(3)] It takes non-negative values on positive operators. 
 
 \item[(4)] It is invariant under Hilbert space isomorphisms.  
 \end{enumerate}
As mentioned above, the infinitesimal operators of order~1 are the operators in the weak trace class $\sL_{1,\infty}$. The condition (3) means that the functional should be positive. The condition (4) forces the functional to be a trace, in the sense it is annihilated by the commutator subspace, 
\begin{equation*}
 \op{Com}\big(\sL_{1,\infty}\big):=\op{Span}\left\{[A,T];\ A\in \sL(\sH), \ T\in\sL_{1,\infty}\right\}. 
\end{equation*}
Therefore, the NC integral must be a positive trace which defined on a suitable subspace of $\sL_{1,\infty}$ containing  and  is  annihilated by $\op{Com}\big(\sL_{1,\infty}\big)$ and infinitesimal operators of order~$>1$. 

\subsection{Construction of the NC integral} 
If $A$ is a compact operator on $\sH$, then its spectrum can be arranged as a sequence $(\lambda_j(A))_{j\geq 0}$ converging to $0$ such that
\begin{equation*}
|\lambda_0(A)|\geq |\lambda_1(A)|\geq \cdots \geq  |\lambda_j(A)|\geq \cdots \geq 0,
\end{equation*}
where eigenvalue is repeating according to its algebraic multiplicity, i.e., the dimension of the corresponding root space 
 $E_\lambda(A) :=\cup_{\ell\geq 0} \ker(A-\lambda)^\ell$. 
If $\lambda\neq 0$, the algebraic multiplicity is always finite (see, e.g., \cite{GK:AMS69}).  It agrees with the dimension of the eigenspace $\ker(A-\lambda)$ when $A$ is norm. 

A sequence as above is called an \emph{eigenvalue sequence}. If the spectrum of $A$ lies on a ray, then $A$ admits a unique eigenvalue sequence. In particular, if $A\geq 0$, then the eigenvalue sequence agrees with its singular value sequence $(\mu_j(T))_{j\geq 0}$. In general, an eigenvalue sequence need not be unique.

In what follows we shall denote by $(\lambda_j(A))_{j\geq 0}$ any eigenvalue sequence for $A$. It is understood that all the results involving eigenvalue sequences do not depend on the choice of the sequence. 

We record the Weyl's inequalities (see, e.g., \cite{GK:AMS69, Si:AMS05}), 
\begin{equation}
\big| \sum_{j<N} \lambda_j(A) \big| \leq  \sum_{j<N} \left|\lambda_j(A) \right|  \leq  \sum_{j<N} \mu_j(A) \qquad \forall N\geq 1. 
\label{eq:NC-Integral.Weyl-Ineq} 
\end{equation}

\begin{lemma}[{\cite[Lemma 5.7.5]{LSZ:Book}}] \label{lem:NCInt.additivity}
 If $A$ and $B$ are operators in $\sL_{1,\infty}$, then
 \begin{equation*}
 \sum_{j<N} \lambda_j(A+B) = \sum_{j<N} \lambda_j(A)  + \sum_{j<N} \lambda_j(B) +\op{O}(1).
\end{equation*}
\end{lemma}

The above result implies the following property of sums of commutators in $\sL_{1,\infty}$. 
\begin{corollary}[\cite{DFWW:AIM04, LSZ:Book}; see also~{\cite[Corollary 2.6]{Po:NCIntegration}}]
 If $A\in \Com(\sL_{1,\infty})$, then
 \begin{equation*}
 \sum_{j<N} \lambda_j(A)=\op{O}(1). 
\end{equation*}
\end{corollary}

\begin{remark}
 Corollary~5.2 has a converse. It can be shown (see~\cite{DFWW:AIM04, LSZ:Book}) that, if $A$ and $B$ are operators in $\sL_{1,\infty}$, then 
 \begin{equation}
 A-B\in  \Com\big(\sL_{1,\infty}\big) \Longleftrightarrow  \sum_{j<N} \lambda_j(A)=\sum_{j<N} \lambda_j(B)+\op{O}(1).
 \label{eq:NCInt.A-B-Com}  
\end{equation}
In particular, this implies that $\Com(\sL_{1,\infty})$ contains the trace class $\sL_1$, and so it contains all infinitesimal operators of order~$>1$. 
\end{remark}

Let $\ell_\infty$ be the $C^*$-algebra of bounded sequences $(a_N)_{N\geq 1}\subset \C$ and $\co$ the closed ideal of sequences converging to $0$. If $A\in \sL_{1,\infty}$, then the Weyl's inequalities~(\ref{eq:NC-Integral.Weyl-Ineq}) imply that  the Ces\`aro mean sequence $\{\frac1{\log N} \sum_{j<N} \lambda_j(A)\}_{N\geq 1}$ is bounded. Moreover, it follows from Lemma~\ref{lem:NCInt.additivity} that its class in $\ell_\infty/\co$ does not depend on the choice of the eigenvalue sequence. 

\begin{definition}\label{def:NCG.NC-Integral}
The \emph{NC integral} of an operator $A\in \sL_{1,\infty}$ is defined by
\begin{equation*}
 \bint A:=  \lim_{N\rightarrow \infty}  \frac{1}{ \log N}\sum_{j<N}\lambda_j(A), 
\end{equation*}
provided the limit exists. An operator for which the above limit exists is called \emph{measurable}. 
\end{definition}

\begin{remark}
The existence and value of the above limit does not depend on the choice of the eigenvalue sequence.  
\end{remark}

 In what follows we denote by $\sM$ the class of measurable operators. The following shows that $\bint$ satisfies the Ansatz for the integral in the quantized calculus.

\begin{proposition}[\cite{Po:NCIntegration}] 
$\sM$ is a subspace of $\sL_{1,\infty}$ containing $\Com(\sL_{1,\infty})$ on which $\bint: \sM\rightarrow \C$ is a positive linear trace which is annihilated by  operators in $(\sL_{1,\infty})_0$, including infinitesimals of order~$>1$. 
\end{proposition}

We stress out that the above construction of the NC integral is of purely spectral nature. In particular, we immediately have the following spectral invariance result. 

\begin{proposition}[\cite{Po:NCIntegration}] 
 Let $A$ and $B$ be weak trace class operators possibly acting on different Hilbert spaces. Then one of these operators is measurable if and only if the other operator is measurable. Moreover, in this case $\bint A=\bint B$. 
\end{proposition}

We have an alternative description of the NC integral in terms of Dixmier traces~\cite{Co:NCG, Di:CRAS66}. Let $\omega$ be a state on the quotient $C^*$-algebra $\ell_\infty/\co$, i.e., a positive linear form such that $\omega(1)=1$. This lifts to a positive linear functional $\lim_\omega$ on $\ell_\infty$. Such a functional is called an \emph{extended limit}. In particular, if $a=(a_N)_{N\geq 1}$ is any bounded sequence, then 
\begin{equation*}
 \lim_{j\rightarrow \infty} a_j = L \ \Longleftrightarrow \ \big( \limw a_j=L \quad \forall \omega\big). 
\end{equation*}

\begin{proposition}[\cite{Di:CRAS66, Co:NCG, LSZ:Book, Po:NCIntegration}] 
 The following formula
\begin{equation*}
 \Trw(A)= \limw \frac{1}{\log N} \sum_{j<N} \lambda_j(A), \qquad A\in \sL_{1,\infty}, 
\end{equation*}
uniquely defines a positive linear trace $\Trw: \sL_{1,\infty}\rightarrow \C$. 
\end{proposition}

The trace $\Trw: \sL_{1,\infty}\rightarrow \C$ is called the \emph{Dixmier trace} associated with the extended limit $\limw$. 

\begin{proposition}[\cite{LSZ:Book, Po:NCIntegration}] \label{prop:NCInt.sM-sT} 
If $A\in \sL_{1,\infty}$, then 
\begin{equation*}
  \lim_{N\rightarrow \infty}  \frac{1}{ \log N}\sum_{j<N}\lambda_j(A)=L  \ \Longleftrightarrow \ \big( \Trw(A)=L  \quad \forall \omega\big).
\end{equation*}
In particular, $A$ is measurable if and only if the value of $\Trw(A)$ is independent of $\omega$. Moreover, in this case this value is equal to the NC integral $\bint A$. 
\end{proposition}

The above proposition shows that we recover the original construction of the NC integral by Connes~\cite{Co:NCG}.

\subsection{Measurability and Weyl's laws} 
In what follows we denote by $T_0$ any positive operator in $\sL_{1,\infty}$ such that $\mu_j(T_0)=(j+1)^{-1}$ for all $j\geq 0$. Note that any two such operators are unitary equivalent, and hence agree up to an element of $\Com(\sL_{1,\infty})$. 

A trace $\varphi:\sL_{1,\infty}\rightarrow \C$ is called \emph{normalized} if $\varphi(T_0)=1$. All the Dixmier traces are normalized traces. However, there are plenty of positive normalized traces on $\sL_{1,\infty}$ that are not Dixmier traces (see, e.g.,~\cite[Theorem~4.7]{SSUZ:AIM15}). Therefore, it stands for reason to consider a stronger notion of measurability. 

\begin{definition}
An operator $A\in \sL_{1,\infty}$ is called \emph{strongly measurable} when there is $L\in \C$ such that $\varphi(A)=L$ for every positive normalized trace $\varphi$ on $\sL_{1,\infty}$. 
 \end{definition}

\begin{remark}\label{rmk:NC.Rozenblum-measurability}
 In~\cite{Ro:arXiv21} measurable operators are defined as operators in the Dixmier-Macaev class $\fM_{1,\infty}$ which take the same value on all positive normalized traces on $\fM_{1,\infty}$. The weak trace class $\sL_{1,\infty}$ is strictly contained in $\fM_{1,\infty}$. 
 Every Dixmier trace extends to a positive normalized trace on $\fM_{1,\infty}$. However, there are many traces on $\sL_{1,\infty}$ that do not extend to  $\fM_{1,\infty}$. Therefore, measurability in the sense of~\cite{Ro:arXiv21} is stronger than measurability in the sense of Definition~\ref{def:NCG.NC-Integral}, 
 but it is also weaker than strong measurability in the sense above. 
\end{remark}

In the scalar case Rozenblum established the trace formula~(\ref{eq:Intro.trace-formula-QuP}) for a larger class of potentials. As mentioned above, a weaker notion of measurability is used in~\cite{Ro:arXiv21}. However, the same arguments as above show that the operators under consideration in~\cite{Ro:arXiv21} are strongly measurable in the sense used in this paper.

We denote by $\sMs$ the class of strongly measurable operators.  It follows from Proposition~\ref{prop:NCInt.sM-sT} that if $A\in \sMs$, then $A$ is measurable, and $\bint A=\varphi(A)$ for every positive normalized trace $\varphi$ on $\sL_{1,\infty}$. 
Note also that the inclusion of $\sMs$ into $\sM$ is strict (see~\cite[Theorem~7.4]{SSUZ:AIM15}).

\begin{proposition}[see, e.g., \cite{Po:NCIntegration}]\label{prop:NCInt.strong-measurable} 
 The following holds. 
\begin{enumerate}
 \item $\sMs$ is a closed subspace of $\sL_{1,\infty}$ containing $\Com(\sL_{1,\infty})$ and $(\sL_{1,\infty})_0$. In particular, it contains all infinitesimal operators of order~$>1$. 
  
\item If $A\in \sL_{1,\infty}$ is such that $\sum_{j<N} \lambda_j( A)=L+\op{O}(1)$, then $A\in \sMs$ and $\bint A=L$.  
  
\item If $A$ and $B$ are operators in $\sL_{1,\infty}$ with same non-zero eigenvalues and same multiplicities. Then one operator is strongly measurable if and only if the other is. 
\end{enumerate}
\end{proposition}

We have the following relationship between strong measurability and Weyl's laws.

\begin{proposition}[see~\cite{Po:NCIntegration}]\label{prop:Bir-Sol.Weyl-mesurable} 
 The following holds. 
\begin{enumerate}
 \item  If $A\in \sW_{1,\infty}$, then $A$ is strongly measurable, and we have
\begin{equation}
 \bint A = \Lambda^+(A)- \Lambda^-(A). 
 \label{eq:Bir-Sol.bint-Lambdapm}
\end{equation}
In particular, if $A$ is selfadjoint, then 
\begin{equation}
 \bint A = \lim_{j\rightarrow \infty}j\lambda_j^+(A) -  \lim_{j\rightarrow \infty}j\lambda_j^-(A).
 \label{eq:Bir-Sol.bint-Lambdapm-sa}
\end{equation}

\item If $A\in  \sW_{|1,\infty|}$, then $|A|$ is strongly measurable, and we have
\begin{equation*}
  \bint |A| = \lim_{j\rightarrow \infty}j\mu_j(A)
\end{equation*}
\end{enumerate}
\end{proposition}

\subsection{Connes' trace theorem}\label{subsec.PsiDOs} 
Suppose that $(M^n,g)$ is a closed Riemannian manifold and $E$ is a Hermitian vector bundle over $M$. Let  $\Res :\Psi^\Z(M,E)\rightarrow \C$ be the noncommutative residue trace of Guillemin~\cite{Gu:AIM85} and Wodzicki~\cite{Wo:LNM87},  where $\Psi^\Z(M,E)=\bigcup_{m\in \Z}\Psi^m(M,E)$ is the algebra of integer order \psidos\ acting on sections of $E$. This functional appears as the residual trace on integer-order \psidos\ induced by the analytic extension of the ordinary trace to all non-integer order \psidos~(see~\cite{Gu:AIM85, Wo:LNM87}). If $P\in \Psi^m(M,E)$, $m\in \Z$, then 
\begin{equation*}
 \Res(P) = \int_M \tr_E[c_P(x)], 
\end{equation*}
where $c_P(x)$ is an $\End(E)$-valued 1-density which is given in local coordinates by
\begin{equation*}
 c_P(x)=(2\pi)^{-n}\int_{|\xi|=1} a_{-n}(x,\xi) d^{n-1}\xi,
\end{equation*}
 where $a_{-n}(x,\xi)$ is the symbol of degree $-n$ of $P$. In particular, if $P$ has order~$-n$, then 
 \begin{equation*}
 \Res(P)= (2\pi)^{-n} \int_{S^*M} \tr_E\big[\sigma(P)(x,\xi)\big] dxd\xi. 
\end{equation*}
A well known result of Wodzicki~\cite{Wo:HDR} further asserts that this is the unique trace on  $\Psi^\Z(M,E)$ up to constant multiple if $M$ is connected and has dimension $n\geq 2$ (see also~\cite{Le:AGAG99, Po:JAM10}). 

\begin{proposition}[Connes's Trace Theorem~\cite{Co:CMP88, KLPS:AIM13}] \label{prop:NCInt.Connes-trace-thm} 
Every operator $P\in \Psi^{-n}(M,E)$ is strongly measurable, and we have
\begin{equation}
 \bint P= \frac{1}{n} \Res(P). 
 \label{eq:Connes-trace-thm} 
\end{equation}
 \end{proposition}

Let $\Delta_E=\nabla^*\nabla$ be the Laplacian of some Hermitian connection $\nabla$ on $E$. As an application of Connes' trace theorem we obtain the following integration formula, which shows that the noncommutative integral recaptures the Riemannian volume density.

\begin{corollary}[Connes's Integration Formula~\cite{Co:CMP88, KLPS:AIM13}]\label{prop:NCInt.Connes-int-formula} 
For all $u\in C^\infty(M, \End(E))$, the operator $ \Delta_E^{-{n}/{4}} u \Delta_E^{-{n}/{4}}$ is strongly measurable, and we have
\begin{equation*}
 \bint \Delta_E^{-\frac{n}{4}} u \Delta_E^{-\frac{n}{4}} = c_n \int_M \tr_E\big[ u(x)\big] \sqrt{g(x)}d^nx, \qquad c_n:=\frac{1}{n} (2\pi)^{-n}|\bS^{n-1}|. 
\end{equation*}
\end{corollary}

We observe that Birman-Solomyak's Weyl's laws~(\ref{eq:Weyl.Bir-Sol-mu})--(\ref{eq:Weyl.Bir-Sol-selfadjoint}) and Proposition~\ref{prop:Bir-Sol.Weyl-mesurable} enable us to recover Connes' trace theorem. In particular, 
if $P=P^*\in \Psi^{-n}(M,E)$, then by using~(\ref{eq:Weyl.Bir-Sol-selfadjoint}) and~(\ref{eq:Bir-Sol.bint-Lambdapm-sa}) we get
\begin{align*}
 \bint P& =  \lim_{j\rightarrow \infty}j\lambda_j^+(P) -  \lim_{j\rightarrow \infty}j\lambda_j^-(P)\\
 & = \frac1{n} (2\pi)^{-n} \! \int_{S^*M} \left\{\tr_E\big[ \sigma(P)(x,\xi)_+ \big]  -\tr_E\big[ \sigma(P)(x,\xi)_{-} \big]\right\} dx d\xi,\\
 &= \frac1{n} (2\pi)^{-n} \! \int_{S^*M} \tr_E\big[ \sigma(P)(x,\xi) \big]  dx d\xi,
\end{align*}
which gives back Connes' trace theorem~(\ref{eq:Connes-trace-thm}). 

In fact, Birman-Solomyak's Weyl's law~(\ref{eq:Weyl.Bir-Sol-mu}) further provide us with a trace formula for absolute values of \psidos\ even though they need not be \psidos. Namely, by combining~(\ref{eq:Weyl.Bir-Sol-mu}) and Proposition~\ref{prop:Bir-Sol.Weyl-mesurable} we get the following result. 

\begin{proposition}
 If $P\in \Psi^{-n}(M,E)$, then $|P|$ is strongly measurable, and we have
\begin{equation*}
 \bint |P|=  
 \lim_{j\rightarrow \infty} j \mu_j(P)= \frac1{n} (2\pi)^{-n} \int_{S^*M} \tr_E\big[ |\sigma(P)(x,\xi)| \big] dx d\xi.
\end{equation*}
\end{proposition}

It is interesting to note that Birman-Solomyak's Weyl's laws predate Connes' trace theorem by a decade. 

\section{Integration Formulas for $\LlogL$-Orlicz Potentials} \label{sec:NCInt-Orlicz} 
In this section, we explain how the Weyl's laws~(\ref{eq:Orlicz.Weyl-|QuP|})--(\ref{eq:Orlicz.Weyl-laws}) allow us to get stronger forms of Connes' trace theorem and Connes's integration formula. 

We keep on using the notation of the previous sections. In particular, we let $(M^n,g)$ be a closed Riemannian manifold and $E$ a Hermitian vector bundle over $M$. 

By combining Theorem~\ref{thm:Orlicz.WeylQuP} and Proposition~\ref{prop:Bir-Sol.Weyl-mesurable} we immediately the following trace theorem. 

\begin{theorem}\label{thm:Orlicz-trace-formula}
Let $P$ and $Q$ be operators in $\Psi^{-n/2}(M,E)$. If $u\in \LlogL(M,\End(E))$, then $QuP$ and $|QuP|$ are both strongly measurable operators, and we have
 \begin{gather}
 \bint QuP = \frac1{n} (2\pi)^{-n} \!  \int_{S^*M} \tr_E \big[ \sigma(Q)(x,\xi)u(x) \sigma(P)(x,\xi)\big]dxd\xi,
 \label{eq:Orlicz.trace-formula}\\
\bint \big|QuP\big| =
 \frac1{n} (2\pi)^{-n} \!  \int_{S^*M} \tr_E \big[\left|\sigma(Q)(x,\xi)u(x) \sigma(P)(x,\xi)\right|\big]dxd\xi.
 \label{eq:Orlicz.trace-formula-|QuP|} 
\end{gather}
 \end{theorem}
 
\begin{remark}
In the scalar case, Rozenblum~\cite{Ro:arXiv21} established obtained a trace formula similar to~(\ref{eq:Orlicz.trace-formula}) 
for a larger class of potentials of the form $u=f\mu$, where  $\mu$ is an Alfhors-regular measure supported on a  regular submanifold $\Sigma \subset M$ and 
$f$ is a real-valued function in $\LlogL(\Sigma,\mu)$. As mentioned in Remark~\ref{rmk:NC.Rozenblum-measurability}, in~\cite{Ro:arXiv21} a weaker notion of measurability is used, but the same arguments as above show that the operators under consideration in~\cite{Ro:arXiv21} are strongly measurable in the sense used in this paper. 
\end{remark}

Specializing the above results to $Q=P=\Delta_E^{-n/4}$ allows us to get the following integration formula. 

\begin{corollary}\label{cor:Orlicz.int-formula-Delta}
If $u\in \LlogL(M,\End(E))$, then $\Delta_E^{-n/4}u\Delta_E^{-n/4}$ and $|\Delta_E^{-n/4}u\Delta_E^{-n/4}|$ are both strongly measurable operators, and we have
\begin{gather}
  \bint \Delta_E^{-\frac{n}{4}}u\Delta_E^{-\frac{n}{4}} = \frac1{n} (2\pi)^{-n} \! \int_M \tr_E\big[u(x)\big] \sqrt{g(x)}dx,
  \label{eq:Orlicz.int-formula-Delta}\\
 \bint \left|\Delta_E^{-\frac{n}{4}}u\Delta_E^{-\frac{n}{4}}\right| =  \frac1{n} (2\pi)^{-n} \! \int_M  \tr_E\big[|u(x)|\big] \sqrt{g(x)}dx.
\end{gather}
\end{corollary}

\begin{remark}
 In the scalar case, Sukochev-Zanin~\cite[Lemma~5.8]{SZ:arXiv21} got a weaker form of the integration formula~(\ref{eq:Orlicz.int-formula-Delta})
  for operators of the form $(1+\Delta_g)^{-n/4}f(1+\Delta_g)^{-n/4}$ with $f\in L_\infty(M)$. We recover this formula from~(\ref{eq:Orlicz.int-formula-Delta}), since the operators $(1+\Delta_g)^{-n/4}f(1+\Delta_g)^{-n/4}$ and $\Delta_g^{-n/4}f\Delta_g^{-n/4}$ agree modulo a trace-class operator. Note that the approach in~\cite{SZ:arXiv21} relies on a deep result of~\cite{DFWW:AIM04} asserting that $\Com(\sL_{1,\infty})$ is spanned by commutators $[A,B]$ with $A,B\in \sL_{2,\infty}$. \end{remark}
  
\begin{remark}\label{rmk:Int-Orlicz.SC-Connes} 
Combining~(\ref{eq:Intro.SC-Weyl}) and~(\ref{eq:Orlicz.int-formula-Delta}) shows that, if $V\in \LlogL(M,\End(E))$, $V(x)\geq 0$, then
\begin{equation*}
  \lim_{h \rightarrow 0^+} h^nN^{-}\big(h^n \Delta_E^{\frac{n}{2}}-V\big)=  \bint \Delta_E^{-\frac{n}{4}}V\Delta_E^{-\frac{n}{4}} = \frac1{n} (2\pi)^{-n} \! \int_M \tr_E\big[V q(x)\big] \sqrt{g(x)}dx. 
\end{equation*}
This gives a semiclassical interpretation of Connes' integration formula. The first equality is also a direct consequence of the fact that  
$ \Delta_E^{-{n}/{4}}u\Delta_E^{-{n}/{4}}$ is a Weyl operator in $\sL_{1,\infty}$ (see~\cite{Po:NCIntegration}; see also~\cite{MSZ:arXiv21}). 
\end{remark}

For applications in noncommutative geometry, as well as in mathematical physics,  it is often preferable to use Dirac-type operators instead of connection Laplacians. As a matter of fact, Connes' original formulation of the integration formula was stated in terms of the Dirac operator (see~\cite[Theorem~3]{Co:CMP88}). Thus, it is worth reformulating Corollary~\ref{thm:Orlicz.int-formulaDelta} and Corollary~\ref{eq:Orlicz.int-formula-Delta} in terms of Dirac operators. 

Assume that $E$ is Clifford module in the sense of~\cite{BGV:Springer92}. Let $\ssD_E:C^\infty(M,E)\rightarrow C^\infty(M,E)$ be the Dirac operator associated with a Hermitian Clifford connection $\nabla$ on $E$. Thus, $\ssD_E$ is a selfadjoint and its square $\ssD_E^2$ has same principal symbol as the corresponding connection Laplacian $\Delta_E$. 

For instance, if $M$ is spin and $E$ is the spinor bundle, then we may take $\ssD_E$ to be the usual Dirac operator $\ssD_M$. If $E$ is the bundle of differential forms on $M$, then we may take $\ssD_E$ to be the de Rham-Dirac operator $d+d^*$, so that $\ssD_E$ is the Hodge-de Rham Laplacian. If $M$ is K\"ahler manifold and $E$ is the bundle of anti-holomorphic forms, then we may take $\ssD_E$ to be the Dolbeault-Dirac operator $\overline{\partial}+\overline{\partial}^*$, in which case $\ssD_E^2$ is the Dolbeault Laplacian. There are numerous examples (see, e.g., \cite{BGV:Springer92}). 

Specializing Theorem~\ref{thm:Orlicz.WeylQuP} and Theorem~\ref{thm:Orlicz-trace-formula} to $P=Q=|\ssD_E|^{-n/2}$ leads to the following statement. 

\begin{proposition}\label{prop:Orlicz.Dirac}
 Under the above assumptions, let  $u\in \LlogL(M,\End(E))$. 
 \begin{enumerate}
\item The operators $|\ssD_E|^{-n/2}u|\ssD_E|^{-n/2}$ and $\big||\ssD_E|^{-n/2}u|\ssD_E|^{-n/2}\big|$  are both Weyl operators in $\sL_{1,\infty}$, and hence are strongly measurable. 

\item We have 
 \begin{gather*}
 \bint |\ssD_E|^{-\frac{n}2}u|\ssD_E|^{-\frac{n}2} = \frac1{n} (2\pi)^{-n} \! \int_M \tr_E\big[u(x)\big] \sqrt{g(x)}dx,\\
 \bint \left| |\ssD_E|^{-\frac{n}2}u|\ssD_E|^{-\frac{n}2}\right|= \lim_{j\rightarrow \infty} j\mu_j \left(|\ssD_E|^{-\frac{n}2}u|\ssD_E|^{-\frac{n}2}\right) 
 = \frac1{n} (2\pi)^{-n} \! \int_M \tr_E\big[|u(x)|\big] \sqrt{g(x)}dx. 
\end{gather*}

\item If $u(x)^*=u(x)$, then 
\begin{equation*}
  \lim_{j\rightarrow \infty} j\lambda^\pm_j\left(|\ssD_E|^{-\frac{n}2}u|\ssD_E|^{-\frac{n}2}\right)= \frac1{n} (2\pi)^{-n} \! \int_M  \tr_E\big[u(x)_\pm\big] \sqrt{g(x)}dx. 
\end{equation*}
\end{enumerate}
\end{proposition}

\begin{remark}
 It would be interesting to obtain analogues of all the results of this paper for (curved) noncommutative tori.  Cwikel-type estimates, CLR-type inequalities, and integration formulas have been established in this setting (see~\cite{MP:Part1, MP:Part2}). In the critical case the results are established for potentials in $L_p$, $p>1$. It is also expected to have a version for the Birman-Solomyak's Weyl's laws for negative order \psidos\ on NC tori (see~\cite{MSZ:arXiv21, Po:Zeta}). Therefore, the main task is to establish Cwikel-type estimates for $\LlogL$-Orlicz potentials on NC tori.
\end{remark}

\end{document}